\documentclass[11pt, oneside, reqno]{article}
\usepackage{amssymb}
\usepackage{amsmath}
\usepackage{amsthm}
\usepackage{amsfonts}
\usepackage{enumerate}
\usepackage{esint,subfig}
\usepackage{latexsym,amssymb}
\usepackage{color}
\usepackage[hidelinks]{hyperref}
\usepackage[latin1]{inputenc}
\usepackage{multicol}
\usepackage[pdftex]{graphicx}
\usepackage[nottoc,notlot,notlof]{tocbibind}

\usepackage{geometry,color,mathtools,fixltx2e,microtype,lastpage,mathrsfs}

\newtheorem{theorem}{Theorem}

\newtheorem{condition}{Condition}

\newtheorem{definition}{Definition}[section]
\newtheorem{lemma}{Lemma}[section]

\newtheorem{proposition}{Proposition}[section]
\newtheorem{remark}{Remark}[section]

\DeclareMathOperator{\Cov}{Cov}
\DeclareMathOperator{\Var}{Var}
\DeclareMathOperator{\Cum}{Cum}
\DeclareMathOperator{\Corr}{Corr}

\allowdisplaybreaks

\begin{document}

\title{Non-Universal Fluctuations of the Empirical Measure for Isotropic Stationary Fields on $\mathbb S^2\times \mathbb R$}
\date{{\footnotesize \today}}
\author{ Domenico Marinucci, Maurizia Rossi, Anna Vidotto}

\maketitle

\vspace{-.9cm}

\begin{abstract}

In this paper, we consider isotropic and stationary real Gaussian random fields defined on $\mathbb S^2\times \mathbb R$ and we investigate the asymptotic 
behavior, as $T\to +\infty$, of the empirical measure (excursion area) in $ \mathbb S^2\times [0,T]$ at any threshold, covering both cases when
the field exhibits short and long memory, {i.e.} integrable and non-integrable
temporal covariance. It turns out that the limiting distribution {is not universal},
depending both on the memory parameters and the threshold. In
particular, in the long memory case a form of Berry's cancellation phenomenon occurs at zero-{level}, {inducing phase transitions for both variance rates and
limiting laws}.

\begin{itemize}
\item \textbf{Keywords and Phrases:} Sphere-cross-time random fields; Empirical measure; Berry's cancellation; Central and non-Central Limit Theorems.

\item \textbf{AMS Classification: } 60G60; 60F05, 60D05, 33C55.
\end{itemize}
\end{abstract}


\section{Introduction}

\subsection{Background and motivations}

In recent years, special interest has been devoted
to the {study} of random fields ${Z=\left\{Z(x),\text{ }x\in \mathbb{S}^{2}\right\} }$ defined on the {two-dimensional unit sphere $\mathbb S^2$, finding applications in several areas such as medical imaging, atmospheric sciences, geophysics, solar physics and cosmology} (see e.g.~\cite{CD:12, Ch:05, MaPeCUP, KN:17}).
In particular, considerable attention has been drawn by the investigation of geometric functionals {of Gaussian excursion sets} on manifolds (see e.g.~\cite{adlertaylor, AW:09}). Indeed, aiming to study the geometry of a random field $Z$, it is natural to introduce {the family} of excursion sets 
\begin{equation*}
\left\{ x\in \mathbb{S}^{2}:Z(x)\geq u\right\}
\end{equation*}
{indexed by the threshold} $u\in \mathbb R$; under Gaussianity and isotropy, the expected value of their Lipschitz-Killing curvatures (i.e.~area, boundary
length and Euler-Poincar\'{e} characteristic), is
easily obtained as a special case of the celebrated Gaussian Kinematic
Formula, see e.g.~\cite[Ch. 13]{adlertaylor}. However, what is more challenging is to investigate
fluctuations around these expected values and for this purpose,
asymptotic methods must be exploited, considering sequences of
random fields. In particular, a number of recent papers has focussed on {the asymptotic behavior of sequences of
Gaussian Laplace eigenfunctions (random spherical harmonics), in the \emph{high-energy}
limit, i.e.~as the eigenvalues diverge}. Several results have
been given concerning the asymptotic
variance, the limiting distribution and the
correlation for different values of the thresholding
parameter $u\in \mathbb{R}$ {of Lipschitz-Killing curvatures of their excursion sets}, see e.g.~\cite{CM2018,  CM19, MR2015,  MRW,  MW2014, Wig, Todino1} and the references therein; see also \cite{Cam19, KKW13, MPRW16, NPR19, PV20} for related results on the standard flat torus and on the Euclidean plane. Some of these results entail rather surprising
issues, for instance the cancellation of the leading variance terms for
specific threshold values and the possibility to express wide classes of
functionals as simple polynomial integrals {on $\mathbb S^2$} of the underlying fields, up to
lower order terms.

The purpose of this paper is to begin the investigation of these same issues
for a different class of fields, i.e., isotropic and stationary Gaussian fields on 
$\mathbb S^2\times \mathbb R$, which can be immediately interpreted as
spherical random fields evolving over time (see e.g.~\cite{BP17, BS19, MM18} and the references therein). Although the present manuscript is
mainly of theoretical nature, it is very easy to figure out several areas
of applications where such random fields emerge most naturally, {including the scientific research streams mentioned above}. In the next
subsection, we introduce our setting in more detail.

\subsection{Sphere-cross-time random fields}

Let us fix a probability space $(\Omega, \mathfrak{F}, \mathbb P)$. We denote by $\mathbb S^2$ the two-dimensional unit sphere with the round metric. 
A space-time real-valued spherical random field
\begin{equation}\label{eq1}
{Z=\{Z(x,t),\,x\in 
\mathbb{S}^{2},\,t\in \mathbb{R}\}}
\end{equation} 
{is a collection, indexed by $\mathbb S^2 \times \mathbb R$, of real random variables} 
such that the map $$Z:\Omega
\times \mathbb{S}^{2}\times \mathbb{R}\rightarrow \mathbb{R}$$ is $\mathfrak{F%
}\otimes \mathfrak{B}(\mathbb{S}^{2}\times \mathbb{R})$-measurable, {where $\mathfrak{B}(\mathbb{S}^{2}\times \mathbb{R})$ stands for the Borel $\sigma$-field of $\mathbb S^2\times \mathbb R$}. 
We say that $Z$ is Gaussian if for every $n\ge 1$, $x_1, \dots, x_n\in \mathbb S^2$, $t_1,\dots t_n\in \mathbb R$,
the random vector $(Z(x_1, t_1), \dots, Z(x_n, t_n))$ is Gaussian. 
\begin{condition}
\label{basic}The space-time real-valued spherical real random field $Z$ in \eqref{eq1} is Gaussian and 
\begin{itemize}
\item zero-mean, i.e.~$\mathbb E[Z(x,t)] = 0$ for every $x\in \mathbb S^2$, $t\in \mathbb R$;
\item stationary and isotropic, i.e.~
\begin{equation}\label{Gamma}
\mathbb E[Z(x,t) Z(y,s)] = \Gamma( \langle x, y \rangle, t-s) 
\end{equation}
for every $x,y\in \mathbb S^2$, $t,s \in \mathbb R$, where $\Gamma:[-1,1]\times \mathbb R \to \mathbb R$ is a positive semidefinite function and $\langle \cdot, \cdot \rangle$ denotes the standard inner product in $\mathbb R^3$;
\item mean square continuous, i.e.~$\Gamma$ is continuous. 
\end{itemize}
\end{condition}
The assumption of zero-mean is of course just a convenient normalization
with no mathematical impact. The assumption of Gaussianity ensures that we
need to make no distinction between so-called weak and strong stationarity, see e.g.~\cite[Definition 5.9]{MaPeCUP}, and it
simplifies some of our proofs to follow; moreover, it is the common
background with basically all the previous literature on the geometry of
excursion sets (starting from \cite{adlertaylor}), likewise the assumption of mean square-continuity, see e.g.~\cite{BP17, LO13} and the references therein.

From now on we assume that $Z$ in \eqref{eq1} satisfies Condition \ref{basic}. 

\subsubsection{Karhunen-Lo\`eve expansions}\label{secKL}

It is well known (see e.g.~ 
\cite[Theorem 3.3]{BP17} or \cite[Theorem 3]{MM18}) that the following expansion for the covariance function $\Gamma$ in \eqref{Gamma} holds:
\begin{equation}\label{eq2} 
\Gamma(\theta, \tau) = \sum_{\ell=0}^{+\infty}  \frac{2\ell+1}{4\pi} C_\ell(\tau) P_\ell(\theta)\,,\quad (\theta, \tau) \in [-1,1] \times \mathbb R\,,
\end{equation} 
where $\lbrace C_\ell, \ell\ge 0\rbrace$ is a sequence of continuous positive semidefinite functions on $\mathbb R$, $P_\ell$ denotes the $\ell$-th Legendre polynomial \cite[\S 4.7]{Sze75} and the series is uniformly convergent, which is equivalent to 
\begin{equation}\label{eqConv}
\sum_{\ell = 0}^{+\infty} \frac{2\ell +1}{4\pi} C_\ell(0) < +\infty\,.
\end{equation}
Obviously $C_\ell(0)\ge 0$ for every $\ell=0,1,2,\dots$. 
Let $T>0$, it is straightforward (see e.g.~\cite{BS19}) to prove that the following Karhunen-Lo\`eve expansion  for $Z$ holds in $L^2(\Omega \times \mathbb S^2 \times [0,T])$: 
 \begin{equation}\label{eqSeries}
 Z(x,t) = \sum_{\ell =0}^{+\infty} \sum_{m=-\ell}^\ell a_{\ell,m}(t) Y_{\ell,m}(x),
 \end{equation}
where $\lbrace Y_{\ell,m}, \ell \ge 0, m=-\ell, \dots, \ell\rbrace$ is the standard real orthonormal basis of spherical harmonics \cite[\S 3.4]{MaPeCUP} for $L^2(\mathbb S^2)$, and 
\begin{equation}\label{eq3}
a_{\ell,m}(t) = \int_{\mathbb S^2} Z(x,t) Y_{\ell,m}(x)\,dx\,,
\end{equation}
so that 
$\lbrace a_{\ell,m}, \ell \ge 0, m=-\ell, \dots, \ell\rbrace$ 
is a family of independent, stationary, centered, Gaussian processes on $\mathbb R$ such that for every $t,s \in \mathbb R$
$$
\mathbb E[a_{\ell,m}(t) a_{\ell,m}(s) ] = C_\ell(t-s).
$$
Now let 
$$\widetilde{\mathbb N} := \lbrace \ell \ge 0 : C_\ell(0)\ne 0\rbrace.$$
From now on, we will consider only $\ell\in \widetilde{\mathbb N}$ unless otherwise specified.  Let us define
\begin{equation}\label{eqZell}
Z_\ell(x,t) := \sum_{m=-\ell}^\ell a_{\ell,m}(t) Y_{\ell,m}(x)\, , \quad (x,t)\in \mathbb S^2\times \mathbb R\,.
\end{equation} 
By construction, $\lbrace Z_\ell, \ell\in \widetilde{\mathbb N}\rbrace$ is 
a sequence of independent random fields and 
each $Z_\ell(\cdot, t)$  almost surely solves the Helmholtz equation 
\begin{equation*}
\Delta _{\mathbb{S}^{2}}Z_{\ell }(\cdot, t)+\ell (\ell +1)Z_{\ell }(\cdot, t)=0\,,
\end{equation*}
where $\Delta_{\mathbb S^2}$ denotes the spherical Laplacian.  
For notational convenience and without loss of generality we also assume
that 
\begin{equation}\label{somma1}
\mathbb{E}\left[ Z^{2}(x,t)\right] =\sum_{\ell\in \widetilde{\mathbb N}}\sigma _{\ell
}^{2}=1\,,\qquad \sigma _{\ell }^{2}:=\mathbb{E}[Z_{\ell
}^{2}(x,t)]=\frac{2\ell +1}{4\pi }C_{\ell }(0)\,.
\end{equation}

\subsubsection{Long and short range dependence}

For $\ell\in \widetilde{\mathbb N}$, Bochner Theorem 
ensures that there exists a probability measure  $\mu_\ell$ on $(\mathbb R, \mathfrak B(\mathbb R))$ such that 
$$
\frac{C_\ell(\tau)}{C_\ell(0)} = \int_{\mathbb R} \text{e}^{i\lambda \tau}\,d\mu_\ell(\lambda)\,,\qquad \tau\in \mathbb R\,.
$$
If $\mu_\ell$ is absolutely continuous with respect to the Lebesgue measure, then we may
 introduce the normalized spectral density as the
function $f_{\ell }:\mathbb{R\rightarrow R}^{+}$ such that
\begin{equation}\label{sp}
\frac{C_\ell(\tau)}{C_\ell(0)} = \int_{\mathbb R
}\text{e}^{i\lambda \tau}f_{\ell }(\lambda )\,d\lambda\,,\qquad \tau\in \mathbb R\,;
\end{equation}%
we have of course
\begin{equation*}
\int_{\mathbb R}f_{\ell }(\lambda )\,d\lambda =1\,.
\end{equation*}%
If $C_\ell$ is integrable on $\mathbb R$, then clearly $f_\ell$ exists. 

Let us now define the family of symmetric real-valued functions $\lbrace g_\beta, \beta\in (0,1]\rbrace$ 
as follows:
\begin{equation}\label{gbeta}
g_\beta(\tau) = \begin{cases} (1 + |\tau|)^{-\beta}\, & \text{if }\,\, \beta\in (0,1)\\
 (1+|\tau|)^{-\alpha}\,& \text{if }\,\, \beta = 1\\
 \end{cases}\,,
\end{equation} 
for some $\alpha \in [2,+\infty)$.

\begin{condition} \label{basic2}
There exists a sequence
$\lbrace \beta_\ell\in (0,1], \ell \in \widetilde{\mathbb N}\rbrace$ such that 
$$
C_\ell(\tau) = G_\ell(\tau) \cdot g_{\beta_\ell}(\tau), \qquad \ell\in \widetilde{\mathbb N},
$$
where $g_{\beta_\ell}$ is as in \eqref{gbeta} and
$$
\sup_{\ell\in \widetilde{\mathbb N}} \left | \frac{G_\ell(\tau)}{C_\ell(0)} - 1\right | =o(1), \quad \text{as }\tau\to +\infty\,.
$$
Moreover $0\in \widetilde{\mathbb N}$ (that is, $C_0(0)\ne 0$) and if $\beta_0=1$ then 
$$
\int_{\mathbb R} C_0 (\tau)\, d\tau > 0\,.
$$
\end{condition}
From now on we assume that Condition \ref{basic2} holds for the sequence $\lbrace C_\ell, \ell\in \widetilde{\mathbb N}\rbrace$. Note that $G_\ell(0)=C_\ell(0)$ for every $\ell\in \widetilde{\mathbb N}$.

\begin{remark}[Abelian/Tauberian type results] \label{rem_sl}
Let $\ell\in \widetilde{\mathbb N}$. The coefficient $\beta_{\ell }$ in Condition \ref{basic2} can be interpreted as a ``memory"
parameter; in particular, for $\beta_{\ell }=1$ (resp.~$\beta_\ell\in (0,1)$)  the covariance function $C_{\ell }$ is
integrable  on $\mathbb R$  (resp.~$\int_{\mathbb R} |C_\ell(\tau)|\,d\tau = +\infty$) and the corresponding process has so-called short (resp.~long) memory behavior. Under some regularity assumptions, an equivalent characterization 
could be given in terms of the behavior at the origin of the spectral density $f_\ell$ in \eqref{sp}: long-memory entailing divergence to infinity, whereas in the short-memory/integrable case $f_\ell$ is immediately seen to be bounded in $0$. 
\end{remark}
\noindent\textbf{Some conventions}. From now on, $c\in (0, +\infty)$ will stand for a universal constant which may change from line to line.
Let $\lbrace a_n, n\ge 0\rbrace$, $\lbrace b_n, n\ge 0\rbrace$ be two sequences of positive numbers: we will write $a_n \sim b_n$ if $a_n/b_n \to 1$ as $n\to +\infty$, $a_n \approx b_n$ whenever 
$a_n/b_n \to c$, $a_n = o(b_n)$ if $a_n/b_n \to 0$, and finally $a_n = O(b_n)$ if eventually $a_n/b_n \le c$. 

\subsection*{Acknowledgements} 

DM and AV acknowledge the MIUR Excellence Department Project awarded to the Department of Mathematics, University of Rome ``Tor Vergata", CUP E83C18000100006.
The research of MR has been supported by the INdAM-GNAMPA Project 2019 \emph{Propriet\`a analitiche e geometriche di campi aleatori} and the ANR-17-CE40-0008 Project \emph{Unirandom}. 

\section{Main results}

Let $u\in \mathbb R$. We consider the random process $\mathcal{A}_u$ on $\mathbb R$ defined as 
\begin{equation}\label{Au}
\mathcal{A}_{u}(t):=\text{area}\left( Z(\cdot, t)^{-1}\left( [u,\infty )\right)
\right) =\int_{\mathbb{S}^{2}}\mathbf{1}_{Z(x,t)\geq u}\,dx,\qquad t\in \mathbb R.
\end{equation}
In words, $\mathcal{A}_{u}(t)$ represents the empirical measure 
(i.e., the excursion area) of $Z(\cdot, t)$ corresponding to the level $u;$ its expected value is
immediately seen to be given by $\mathbb{E}\left[ \mathcal{A}_{u}(t)\right]
=4\pi (1-\Phi (u)),$ where 
\begin{equation*}
\Phi (u):=\int_u^{+\infty}\phi(t)\,dt, \qquad \phi(t) :=\frac{1}{\sqrt{2\pi }}\text{e}^{-t^2/2},
\end{equation*}%
$\Phi$ (resp.~$\phi$) denoting the tail distribution (resp.~probability density) function of a standard Gaussian random 
variable. 

We are interested in the fluctuations of $\mathcal{A}_{u}(t)$
around its expected value, and we hence introduce the following statistics: for $T>0$
\begin{equation}\label{rap_int}
\mathcal{M}_{T}(u):=\int_{[0,T]}\Big(\mathcal{A}_{u}(t)-\mathbb{E}[%
\mathcal{A}_{u}(t)]\Big)\,dt
\end{equation}%
and its normalized version 
\begin{equation}
\widetilde{\mathcal{M}}_{T}(u):=\frac{\mathcal{M}_{T}(u)}{\sqrt{\Var%
\mathcal{M}_{T}(u)}}\text{ .}  \label{eq:stdM}
\end{equation}
\begin{condition}\label{condbeta}
Let $\lbrace \beta_\ell, \ell \in \widetilde{\mathbb N} \rbrace$ be the sequence defined in Condition \ref{basic2}. 
\begin{itemize}
\item The sequence $\lbrace \beta_\ell, \ell \in \widetilde{\mathbb N}, \ell\ge 1\rbrace$ admits minimum. Let us set 
\begin{equation}\label{def_star}
\beta_{\ell ^{\star }}:=\min\lbrace \beta _{\ell }, \ell \in \widetilde{\mathbb N},\ell\ge 1\rbrace,\qquad \mathcal{I%
}^{\star }:=\{\ell\in \widetilde{\mathbb N} :\beta _{\ell }=\beta _{\ell ^{\star }}\}.
\end{equation}
\item If $\mathcal I^\star\ne \widetilde{\mathbb N}$, then the sequence 
$\lbrace \beta_\ell, \ell\in \widetilde{\mathbb N} \setminus \mathcal I^\star,\ell\ge 1\rbrace$ admits minimum. Let us set 
\begin{equation}\label{def_starstar}
\beta_{\ell ^{\star \star }}:=\min \left\{ \beta _{\ell },\  \ell \in 
\mathbb{N}\backslash \mathcal{I}^{\star },\ell\ge 1\right\}.
\end{equation}
\end{itemize}
\end{condition} 
Note that $\beta_{\ell^\star}, \beta_{\ell^{\star\star}}\in (0,1]$ and for $\ell\in \mathcal I^\star$, obviously $C_{\ell}(0)> 0$. In words, $\beta _{\ell ^{\star }}$ represents the smallest exponent
corresponding to the largest memory, $\mathcal{I}^{\star }$ the set of
multipoles where this minimum is achieved, and
$\beta _{\ell ^{\star \star }}$ 
the second smallest exponent $\beta _{\ell }$ governing the time decay of
the autocovariance $C_\ell$ at some given multipole $\ell$. Note that we are
\textit{excluding} the multipole $\ell =0$ by the definition of $\beta_{\ell^\star}$ and $\beta_{\ell^{\star\star}}$ in \eqref{def_star} and \eqref{def_starstar}, 
on the other 
hand $\ell=0$ may belong to  $\mathcal I^\star$. 
We assume that Condition \ref{condbeta} holds from now on. 

 As we shall see
below, the asymptotic behavior of ${\mathcal{M}}_{T}(u)$ in \eqref{rap_int}, as $T\rightarrow +\infty $, is governed by a subtle interplay between the value of
the parameters $\beta _{\ell ^{\star}}$, $\beta _{\ell ^{\star \star }}$ and the threshold level $u$. 

\subsection{Long memory behavior}\label{sec_lm}

We start investigating the case of long-range dependence.
\begin{theorem}\label{uno} 
If either $u\ne 0$ and $\beta _{0}<\min (2\beta _{\ell ^{\star }},1)$ or $u= 0$ and $\beta _{0}<\min (3\beta _{\ell ^{\star }},1)$, then
\begin{equation*}
\lim_{T\rightarrow \infty }\frac{\Var(\mathcal{M}_{T}(u))}{T^{2-\beta
_{0}}}=\frac{2\phi ^{2}(u)\,C_{0}(0)}{(1-\beta_0)(2-\beta_0)},
\end{equation*}%
and%
\begin{equation*}
\widetilde{\mathcal{M}}_{T}(u)\mathop{\rightarrow}^d _{T\to +\infty} Z,
\end{equation*}
where $Z\sim \mathcal{N}(0,1)$ is a standard Gaussian random variable and $\mathop{\rightarrow}^d$ denotes convergence in distribution. 
\end{theorem}
Recall that by assumption $C_0(0)>0$ (see Condition \ref{basic2}), hence the limiting variance constant in Theorem \ref{uno} is strictly positive. 
\begin{remark}
In words, Theorem \ref{uno} holds when the zero-th order multipole component 
$\left\{ a_{00}(t), t\in\mathbb{R}\right\} $ is long memory ($\beta _{0}<1$) and
all the other multipoles have asymptotically smaller variance (a consequence of
either $\beta _{0}<2\beta _{\ell ^{\star }}$, when $u\ne0$, or $\beta _{0}<3\beta _{\ell ^{\star }}$, when $u=0$, as we will show below). It should
be recalled that, by \eqref{eq3},
\begin{equation*}
a_{00}(t)=\int_{\mathbb{S}^{2}}Z(x,t) Y_{00}(x)\,dx=\frac{1}{\sqrt{%
4\pi }}\int_{\mathbb{S}^{2}}Z(x,t)\,dx\,,
\end{equation*}%
that is, $a_{00}(t)$ corresponds to the sample mean of the random field $Z(\cdot, t)$ at the instant $t\in \mathbb{R}.$
\end{remark}
The limiting distribution in Theorem \ref{uno} is universal; this is not the
case for the theorems to follow. We need first to recall one more definition.
\begin{definition}
The random variable $X_\beta$ has the standard\footnote{Indeed $\mathbb E[X_\beta]=0$ and $\Var(X_\beta)=1$} Rosenblatt
distribution (see e.g.~\cite{DM79})  with parameter $\beta \in (0,\frac{1}{2})$ if it can be written
as 
\begin{equation}\label{Xbeta}
X_{\beta}=a(\beta) \int_{(\mathbb{R}^{2})^{^{\prime }}}\frac{e^{i(\lambda
_{1}+\lambda _{2})}-1}{i(\lambda _{1}+\lambda _{2})}\frac{W(d\lambda
_{1})W(d\lambda _{2})}{|\lambda _{1}\lambda _{2}|^{(1-\beta )/2}}\text{ ,}
\end{equation}%
where $W$ is a white noise Gaussian measure on $\mathbb{R}$, the
stochastic integral is defined in the Ito's sense (excluding the diagonals\footnote{$(\mathbb R^2)'$ stands for the set $\lbrace (\lambda_1, \lambda_2)\in \mathbb R^2:\lambda_1 \ne \lambda_2\rbrace$}), and 
\begin{equation}\label{abeta}
a(\beta) := \frac{\sigma(\beta)}{2\,\Gamma(\beta)\,\sin\left({(1-\beta)\pi}/{2}\right)}\,,
\end{equation}
with 
$$
\sigma(\beta) := \sqrt{\frac12(1-2\beta)(1-\beta)}\,.
$$
We say the random vector $V$ satisfies a composite Rosenblatt distribution
of degree $N\in \mathbb N$ with parameters $c_{1},...,c_{N}\in \mathbb R,$ if%
\begin{equation}\label{V}
V=V_{N}(c_{1},...,c_{N};\beta )\mathop{=}^d\sum_{k=1}^{N}c_{k}X_{k;\beta}%
\text{ ,}
\end{equation}%
where $\left\{ X_{k;\beta}\right\} _{k=1,...,N}$ is a collection of i.i.d.
standard Rosenblatt random variables of parameter $\beta $. 
\end{definition}

\begin{remark}
The characteristic function $\Xi_V$ of $V=V_{N}(c_{1},...,c_{N};\beta )$ in \eqref{V} is given by (see e.g.~\cite{VT13})
\begin{equation*}
\Xi_V(\theta) = \prod_{k=1}^N \xi_\beta(c_k \theta),\qquad \xi_\beta(\theta)= \exp\left (  \frac12 \sum_{j=2}^{+\infty} \left ( 2i \theta \sigma(\beta)  \right )^j  \frac{a_j}{j}  \right ),
\end{equation*}
where $\xi_\beta$ is the characteristic function of $X_\beta$ in \eqref{Xbeta}, the series is only convergent near the origin and 
$$
a_j := \int_{[0,1]^j} |x_1 - x_2|^{-\beta} |x_2 - x_3|^{-\beta}\cdots |x_{j-1} - x_j|^{-\beta} |x_j - x_1|^{-\beta}dx_1 dx_2 \cdots dx_j .
$$
Note that when $\beta\to 0^+$ then $\xi_\beta$ approaches the characteristic function of $\frac{1}{\sqrt 2}(Z^2 -1)$, where $Z\sim \mathcal N(0,1)$ is a standard Gaussian random variable. As $\beta \to \frac12^-$ the limit is the characteristic function of $Z$. 
\end{remark}

\begin{theorem}
\label{due} Assume that $u\neq 0$. If $2\beta _{\ell ^{\star }}<\min
(\beta _{0},1)$ we have 
\begin{equation*}
\lim_{T\rightarrow \infty }\frac{\Var\left( \mathcal{M}_{T}(u)\right) }{%
T^{2-2\beta _{\ell ^{\star }}}}=\frac{u^2\phi (u)^{2}}{2(1-2\beta
_{\ell ^{\star }})(1-\beta _{\ell ^{\star }})}\sum_{\ell \in \mathcal{I}%
^{\star }}(2\ell +1)C_{\ell }(0)^2.
\end{equation*}%
If $\beta_0=1$ and  $2\beta_{\ell^\star}=1$ we have 
$$
\lim_{T\rightarrow \infty }\frac{\Var\left( \mathcal{M}_{T}(u)\right) }{
T\log T} = u^{2}\phi (u)^{2}
\sum_{\ell \in \mathcal{I}^{\star }}(2\ell +1)\,C_{\ell}(0)^2.
$$
Assume in addition that $\#\mathcal{I}^{\star }$ is finite, then as $T\to +\infty$
\begin{equation}\label{limitR}
\widetilde{\mathcal{M}}_{T}(u) \mathop{\rightarrow}^d \sum_{\ell\in \mathcal{I}^\star}\frac{C_\ell(0)}{\sqrt{v^\star}}V_{2\ell+1}(1,\dots,1;\beta
_{\ell ^{\star }}),
\end{equation}%
where $\lbrace V_{2\ell+1}(1,\dots,1;\beta_{\ell ^{\star }}), \ell\in \mathcal I^\star\rbrace $ is a family of independent composite Rosenblatt random variables as in \eqref{V}  and
\begin{eqnarray*}
v^\star=a(\beta_{\ell^\star})^2\sum_{\ell\in \mathcal{I}^\star} \frac{2\,(2{\ell}+1)C_{\ell}(0)^2}{(1-\beta_{\ell^\star})(1-2\beta_{\ell^\star})},
\end{eqnarray*}
where $a(\beta_{\ell^\star})$ is as in \eqref{abeta}. 
\end{theorem}
Recall that for $\ell\in \mathcal I^\star$ we have $C_{\ell}(0) >0$ (see \eqref{def_star}) hence the limiting variance constants in Theorem \ref{due} are strictly positive. 
For the limiting random variable in \eqref{limitR}, note that 
$$
\sum_{\ell\in \mathcal{I}^\star}\frac{C_\ell(0)}{\sqrt{v^\star}}V_{2\ell+1}(1,\dots,1;\beta
_{\ell ^{\star }}) \mathop{=}^d V_{N^\star}(c_1,\dots,c_{N^\star};\beta_{\ell^\star})\,,
$$
where $N^\star:=\sum_{\ell\in \mathcal{I}^\star}(2{\ell}+1)$ and 
$$
(c_1,\dots,c_{N^\star})=\frac{1}{\sqrt {v^\star}}(\underbrace{C_{\ell_1}(0),\dots,C_{\ell_1}(0)}_{(2\ell_1+1) \,\text{ times }},\dots,\underbrace{C_{\ell_k}(0),\dots,C_{\ell_k}(0)}_{(2\ell_k+1) \,\text{ times }})\,,
$$
with $\mathcal{I}^\star=\left\{\ell_1,\dots,\ell_k\right\}$.
\begin{remark}[Normal approximation of Rosenblatt distributions] The distribution of the random variable 
in \eqref{V} is, of course, non-Gaussian. However, in some
circumstances it can be closely approximated by a Normal law. Indeed, consider for
simplicity the case where the minimum for $\lbrace \beta _{\ell}, \ell \in \widetilde{\mathbb N} \rbrace$ is attained in a
single multipole that we call $\ell^\star$, i.e., $\mathcal{I}^{\star }=\lbrace \ell^\star\rbrace$. 
Then the limiting distribution in \eqref{limitR} is 
$$
\frac{C_{\ell^\star}(0)}{\sqrt{v^\star}}V_{2\ell^\star+1}(1,\dots,1;\beta_{\ell^*})=\frac{1}{\sqrt{2\ell^\star+1}}\,V_{2\ell^\star+1}(1,\dots,1;\beta_{\ell^*})
$$ 
and by an immediate application of the classical Berry-Esseen Theorem
(see e.g.~\cite{Ess42}) one has that
\begin{equation*}
d_{Kol}\left (\frac{1}{\sqrt{2\ell^\star+1}}V_{2\ell ^{\star }+1}(1,...,1,\beta _{\ell ^{\star }}), Z\right )\leq c\cdot \frac{\mathbb{E}\left [|X_{\beta_{\ell^\star}}|^3\right ]}{\sqrt{2\ell ^{\star
}+1}},
\end{equation*}
where $Z\sim \mathcal N(0,1)$ and $d_{Kol}$ denotes Kolmogorov distance, see e.g.~\cite[\S C.2]{noupebook}. 
The value of $\ell ^{\star }$ for a given random field is fixed, so no
Central Limit Theorem occurs; however for $\ell ^{\star }$ large enough the
resulting composite Rosenblatt distribution can become arbitrary close to a
standard Gaussian variable.
\end{remark}

\begin{theorem}
\label{tre} Assume that $u=0$ and that there exists an even\footnote{The motivation for this assumption is described just after the statement of Theorem \ref{tre}.} multipole $\ell\in \mathcal I^\star$. If $3\beta _{\ell ^{\star }}<\min (1,\beta _{0})$, then
\begin{eqnarray*}
&&\lim_{T\rightarrow \infty }\frac{\Var\left( \mathcal{M}_{T}(u)\right) 
}{T^{2-3\beta _{\ell ^{\star }}}}\\
&&=\frac{2}{3!(1-3\beta_{\ell^\star})(2-3\beta_{\ell^\star})}\sum_{\ell
_{1},\ell _{2},\ell _{3}\in \mathcal{I}^\star} \mathcal{G}_{\ell _{1}\ell _{2}\ell _{3}}^{000}\prod_{i=1}^{3}\sqrt{\frac{2\ell _{i}+1}{4\pi }} C_{\ell_i}(0),
\end{eqnarray*} 
where 
\begin{equation}\label{gaunt3}
\mathcal{G}_{\ell _{1}\ell _{2}\ell _{3}}^{000}:=\int_{\mathbb{S}^2}Y_{\ell _{1}, 0}(x) Y_{\ell_2,0}(x) Y_{\ell _{3}, 0}(x) \,dx
\end{equation}
 is a so-called Gaunt integral 
(cf. \eqref{eq:gaunt}. 
If $\beta_0=1$ and $3\beta_{\ell^\star} =1$ then 
$$
\lim_{T\rightarrow \infty }\frac{\Var\left( \mathcal{M}_{T}(u)\right) 
}{T\log T} = \frac{8\pi\,H_{q-1}(u)^2\phi(u)^2}{q!}\sum_{\ell
_{1},\ell _{2},\ell _{3}\in \mathcal{I}^\star} \mathcal{G}_{\ell _{1}, \ell_2, \ell _{3}}^{000} \prod\limits_{i=1}^{3}\sqrt{\frac{2\ell _{i}+1}{4\pi }} C_{\ell_i}(0)\,.
$$
Moreover we have, as $T\to +\infty$, 
\begin{equation}\label{conv3}
\widetilde{ \mathcal M}_T(u) = -  \frac{\Var(\mathcal M_T(u))^{-1/2}}{3!\sqrt{2\pi}}\int_{\mathbb S^2\times [0,T]} H_3(Z(x,t))\,dxdt + o_{\mathbb P}(1),
\end{equation}
where $H_3(t):=t^3 -3t$, $t\in \mathbb R$ is the third Hermite polynomial (cf. \eqref{Herm}) and $o_{\mathbb P}(1)$ is a family of random variables converging to zero in probability. 
\end{theorem}
Recall that for $\ell\in \mathcal I^\star$ we have $C_{\ell}(0)>0$ (see \eqref{def_star}) and note that $\mathcal{G}_{\ell _{1}\ell _{2}\ell _{3}}^{000}$  in \eqref{gaunt3} is nonnegative (see e.g.~\cite[Remark 3.45]{MaPeCUP}); moreover 
if $\ell$ is even, then $\mathcal{G}_{\ell \ell \ell }^{000} >0$ \cite[Proposition 3.43, (3.61)]{MaPeCUP}. Hence under the assumptions of Theorem \ref{tre} the limiting variance constants are strictly positive. 

\begin{remark}\label{rem_conj}
Using the same steps as in the classical papers \cite{DM79, Ta:79}, it seems possible to prove that the right hand side of \eqref{conv3} (and hence $\widetilde{\mathcal M}_T(u)$), under the setting of Theorem \ref{tre}, converges in distribution to a weighted sum of higher order Rosenblatt random variables (more precisely, of order 3). However, because for the probability laws of the latter very little is known, and even less so for their linear combinations, we refrain from rigorously investigating this issue here.
\end{remark}

\begin{remark}
For simplicity of presentation, we are ruling out some boundary cases (such
as $\beta _{0}=2\beta _{\ell ^{\star }}$), which could be dealt with the same
techniques as we shall exploit below: the limit distributions would just
correspond to linear combinations of the asymptotic random variables that we
obtained above.
\end{remark}

\subsection{Short memory behavior}

Theorems \ref{uno}, \ref{due} and \ref{tre} have all considered cases where some form of long-memory behavior is present on the temporal side, meaning that $\beta
_{\ell }<1$ for at least one instance of the multipole $\ell .$ In this section we investigate 
the case where on all scales no form of long-range dependence occurs.

We first need to introduce some more notation: for $q\ge 3$, let $\ell_1, \dots, \ell_q \ge 0$ and 
$m_i \in \lbrace -\ell_i, \dots, \ell_i\rbrace$ for  $i=1, \dots, q$. 
The \emph{generalized Gaunt integral}  \cite[p. 82]{MaPeCUP} of parameters $q, \ell_1, \dots, \ell_q, m_1, \dots, m_q$ is defined as (cf. \eqref{gaunt3})
\begin{equation}\label{eq:gaunt}
\mathcal{G}_{\ell _{1}...\ell _{q}}^{m_1...m_q}:=\int_{\mathbb{S}^2}Y_{\ell _{1}, m _{1}}(x)\cdots Y_{\ell _{q}, m_{q}}(x) \,dx\,,
\end{equation}
where $\lbrace Y_{\ell,m}, \ell\ge 0, m=-\ell,\dots, \ell\rbrace$ still denotes the family of spherical harmonics introduced in \S \ref{secKL}.

\begin{theorem}
\label{quattro} Assume $\beta_0=1$. If either $u\ne0$ and $2\beta _{\ell^\star}>1$ or $u=0$ and $3\beta _{\ell^\star}>1$, we have
\begin{equation*}
\lim_{T\rightarrow \infty }\frac{\Var\left( \mathcal{M}_{T}(u)\right) 
}{T}=\sum_{q=1}^{+\infty} s_q^2\, ,
\end{equation*}
where 
\begin{eqnarray}\label{def_s}
s_1^2 &:=& \phi(u)^2 \int_{\mathbb R} C_0(\tau)\,d\tau,\notag \\
s^2_2 &:=& \frac{u^2\phi(u)^2}{2} \sum_{\ell=0}^{+\infty} (2\ell+1) \int_{\mathbb R} C_\ell(\tau)^2\,d\tau,\\
s^2_q &:=&  \frac{4\pi\,H_{q-1}(u)^2\phi(u)^2}{q!} \sum_{\ell_1, \dots, \ell_q=0}^{+\infty} \mathcal G_{\ell_1 \dots \ell_q}^{0\dots 0} \int_{\mathbb R} \prod_{i=1}^q \sqrt{\frac{2\ell_i +1}{4\pi}} C_{\ell_i}(\tau)\,d\tau, \ q\ge 3. \notag
\end{eqnarray}
Moreover, as $T\to +\infty$, 
\begin{equation*}
\widetilde{\mathcal{M}}_{T}(u)\stackrel{d}{\rightarrow} Z,
\end{equation*}
$Z\sim \mathcal N(0,1)$ being a standard Gaussian random variable.
\end{theorem}
Recall that for $\beta_0=1$ we have $\int_{\mathbb R} C_0(\tau)\,d\tau \in (0,+\infty)$ (see Condition \ref{basic2}) so that $s_1^2>0$ yielding $\sum_{q\ge 1}s^2_q >0$
 (the limiting variance constant is strictly positive). Moreover from \eqref{def_star} we have that
$C_\ell(0) >0$ for $\ell\in \mathcal I^\star$, and 
we will see (\eqref{Gaunt+} and \eqref{Gaunt2} that $\mathcal{G}_{\ell _{1}...\ell _{q}}^{0\dots 0}\ge 0$ and $s^2_q\ge 0$. 

\begin{remark}[On Berry's cancellation] \label{rem_berry} It is interesting to note that a phase transition occurs at $u=0$. Indeed, for $2\beta _{\ell ^{\star }}<1$ one observes a form of \emph{Berry's cancellation phenomenon} (see e.g.~\cite{Ber02, Wig}), in
the sense that the variance diverges with a smaller order rate. More precisely, there are two
possibilities:

\begin{itemize}
\item for $3\beta _{\ell ^{\star }}<1$ (resp.~$3\beta _{\ell ^{\star }}=1$), the rate of the variance changes from $%
T^{2-2\beta _{\ell ^{\star }}}$ to $T^{2-3\beta _{\ell ^{\star }}},$ (resp.~$T\log T$) and the
limiting distribution is nonGaussian (Theorem \ref{tre} and Remark \ref{rem_conj});
\item  for $3\beta _{\ell ^{\star }}>1$, the rate of the variance
changes from $T^{2-2\beta _{\ell ^{\star }}}$ to $T,$ and the limiting
distribution is Gaussian (Theorem \ref{quattro}).
\end{itemize}

\end{remark}

\section{Outline of the paper}

The results in Theorems \ref{uno}, \ref{due}, \ref{tre} and \ref{quattro} fully characterize the
behavior of the empirical measure for sphere-cross-time random fields. The resulting scheme is, in the end, rather simple and
can be summarized as follows.
\begin{itemize}
\item \emph{Short Memory Behavior}: this setting corresponds to
integrable covariance functions and occurs either when $\beta_0=1$ and $2\beta_{\ell^\star}>1$, for $u\ne0$, or when $\beta_0=1$ and $3\beta_{\ell^\star}>1$, for $u=0$. In such circumstances, the limiting
distribution is always Gaussian and the variance, as $T\to +\infty$, is asymptotic to $T,$ for
all values of the threshold parameter $u.$ Hence, no form of Berry's cancellation, as in Remark \ref{rem_berry}, can occur.
\item 
\emph{Long Memory Behavior}: this setting corresponds to non-integrable temporal autocovariance and in
this case the picture is more complicated:
\begin{itemize}
\item for $\beta _{0}<\min (2\beta _{\ell^\star},1)$, the variance
grows as $T^{2-\beta _{0}}$ and the limiting distribution is Gaussian, for
all values of $u$;
\item for $2\beta _{\ell ^{\star }}<\min (\beta _{0},1)$, the variance grows as $T^{2-2\beta _{\ell ^{\star }}}$ and the
limiting distribution is nonGaussian (we denote it as composite Rosenblatt), for $u\neq 0;$
however, for $u=0$, a form of Berry's cancellation occurs, the variance is of
order $T^{\max (2-3\beta _{\ell ^{\star }},1)}$, the limiting distribution being nonGaussian for 
$2-3\beta _{\ell ^{\star }}>1$ and Gaussian for $2-3\beta _{\ell ^{\star }}<1$.
\end{itemize}
\end{itemize} 

\subsection{Overview of the proofs}

The rationale behind these results can be more easily understood if we
review the main ideas behind the proof. 

\subsubsection{Chaotic expansions}

The main technical
tool that we are going to exploit is the possibility to expand our area
functional $\mathcal M_T(u)$ in \eqref{rap_int}  into so-called Wiener chaoses, by means of the Stroock-Varadhan
decomposition, see \cite[\S 2.2]{noupebook} as well as our \S \ref{sec-Wiener}. Briefly, the latter is based on the fact that the sequence 
of (normalized) Hermite polynomials $\lbrace H_q/\sqrt{q!}\rbrace_{q\ge 0}$ 
\begin{equation}\label{Herm} 
H_0\equiv 1,\qquad H_{q}(u):=(-1)^{q}\phi (u)^{-1}\frac{d^q}{du^q}\phi(u), \ q\ge 1
\end{equation} 
(where $\phi$ still denotes the probability density function of a standard Gaussian random variable)  is a complete orthonormal basis  of the space of square integrable functions on the real line with 
respect to the Gaussian measure. The first three polynomials are $H_0(u)=1$, $H_1(u)=u$, $H_2(u)=u^2-1$, $H_3(u) = u^3 - 3u$.

From \eqref{rap_int} we have 
 the following \emph{orthogonal} expansion
\begin{equation}\label{exp_s}
\mathcal{M}_{T}(u)=\sum_{q=0}^{\infty }\mathcal{M}_{T}(u)[q],
\end{equation}
the series converging in $L^2(\Omega)$, 
where (see Lemma \ref{lem_chaos})
\begin{equation}\label{def_q}
\mathcal{M}_{T}(u)[q]=\frac{H_{q-1}(u)\phi (u)}{q!}%
\int_{[0,T]}\int_{\mathbb{S}^{2}}H_{q}(Z(x,t))\,dxdt
\end{equation}
is the orthogonal projection of $\mathcal M_T(u)$ onto the so-called $q$-th Wiener chaos. Note that if $u=0$, then 
$\mathcal M_T(u)[q]=0$ whenever $q$ is even. 

 In particular, the zeroth projection is 
\begin{equation}\label{1mean}
\mathcal{M}_{T}(u)[0]=\mathbb{E}\left[ \mathcal{M}_{T}(u)%
\right] =0\,; 
\end{equation}
for the first one we have, recalling \eqref{eqSeries} and \eqref{eqZell}, 
\begin{flalign*}
\mathcal{M}_{T}(u)[1] = \phi(u)\int_{[0,T]}\int_{\mathbb{S}^{2}}Z(x,t)\,dxdt = \phi(u)\lim_{L\rightarrow \infty}\int_{[0,T]}\int_{\mathbb{S}^{2}}\sum_{\ell =0}^{L}Z_{\ell }(x,t)\,dxdt,
\end{flalign*}
where the limit is in the $L^2(\Omega)$-sense. Hence 
\begin{eqnarray}
\mathcal{M}_{T}(u)[1]=\phi (u)\int_{[0,T]}\frac{a_{00}(t)}{\sqrt{4\pi }}\,dt,\label{eq1-discussion}
\end{eqnarray}
the spherical harmonics of degree $\ell\ge 1$ having zero mean 
on the sphere. Furthermore
\begin{eqnarray*}
\mathcal{M}_{T}(u)[2] &=&\frac{u\phi (u)}{2}\int_{[0,T]}\int_{%
\mathbb{S}^{2}}\left ( Z^{2}(x,t)-1\right ) dxdt, \\
\mathcal{M}_{T}(u)[3] &=&\frac{(u^{2}-1)\phi (u)}{2}%
\int_{[0,T]}\int_{\mathbb{S}^{2}}\left ( Z^{3}(x,t)-3Z(x,t)\right ) dxdt.
\end{eqnarray*}

\subsubsection{Sharp asymptotics}

The crucial step behind our arguments is to investigate the \emph{sharp} asymptotic
behavior, as $T\to +\infty$, of the variances for these chaotic projections. In order to simplify this
 discussion we assume here that $\beta_{\ell^\star}\le \beta_0$, see the next sections for a complete analysis. 
 For every $u\in \mathbb R$,
\begin{equation}\label{var123}
\Var(\mathcal{M}_{T}(u)[1]) =  c_1\cdot T^{\max (2-\beta _{0},1)}(1+o(1)). 
\end{equation} 
For $q\ge2$ and either $u\ne 0$ or $u=0$ and $q$ odd (recall that for $u = 0$ the projections onto even order chaoses vanish), we have
\begin{equation}\label{varq}
\Var\left ( \mathcal{M}_{T}(u)[q]\right ) = c_q\cdot T^{\max (2-q\beta _{\ell ^{\star }},1)}(1 + \mathbf{1}_{q\beta_{\ell^\star}=1}\cdot\log T)(1+o(1)).
\end{equation}
Here, for $q\ge1$, $c_q=c_q(u, \beta_{\ell^\star}, \mathcal I^\star)$ is a finite and positive constant depending in particular on $q$, the level $u$ and the coefficient $\beta_{\ell^\star}$. 
Thanks to \eqref{exp_s} and \eqref{1mean}, 
\begin{equation*}
\Var\left ( \mathcal{M}_{T}(u)\right ) =\sum_{q=1}^{\infty }\Var\left (\mathcal{M}_{T}(u)[q]\right )
\end{equation*}%
and hence, up to controlling the sequence $\lbrace c_q, q\ge 1\rbrace$, from \eqref{var123} and \eqref{varq} we have that, as $T\rightarrow \infty$,
\begin{eqnarray*}
\widetilde{\mathcal{M}}_{T}(u) &=&\frac{\mathcal{M}_{T}(u)[1]}{%
\sqrt{\Var\left ( \mathcal{M}_{T}(u)[1]\right ) }}+o_{\mathbb{P}}(1),\quad \text{for }\beta _{0}<\min (2\beta _{\ell ^{\star }},1)\text{ , }u\neq 0, \\
\widetilde{\mathcal{M}}_{T}(u) &=&\frac{\mathcal{M}_{T}(u)[1]}{%
\sqrt{\Var\left ( \mathcal{M}_{T}(u)[1]\right ) }}+o_{\mathbb{P}}(1),\quad \text{for }\beta _{0}<\min (3\beta _{\ell ^{\star }},1)\text{ , }u= 0, \\
\widetilde{\mathcal{M}}_{T}(u) &=&\frac{\mathcal{M}_{T}(u)[2]}{%
\sqrt{\Var\left ( \mathcal{M}_{T}(u)[2]\right ) }}+o_{\mathbb{P}}(1),\quad \text{for }2\beta _{\ell ^{\star }}<\min (\beta _{0},1)\text{ , }u\neq 0, \\
\widetilde{\mathcal{M}}_{T}(u) &=&\frac{\mathcal{M}_{T}(u)[3]}{%
\sqrt{\Var\left ( \mathcal{M}_{T}(u)[3]\right ) }}+o_{\mathbb{P}}(1),\quad \text{for }3\beta _{\ell ^{\star }}<\min (\beta _{0},1)\text{ , }u=0,
\end{eqnarray*}
where $o_{\mathbb P}(1)$ denotes a sequence of random variables converging to zero in probability. 
The asymptotic distribution of \eqref{eq:stdM} can then be derived in the cases considered just above by a careful
analysis of these single components: the first chaotic term is Gaussian for every $T>0$, the second one asymptotically follows a composite Rosenblatt distribution. 
 On the other hand, in the remaining
cases, (e.g.~$u\neq 0$, $\beta_0=1$ and $2\beta _{\ell ^{\star }}>1$ or $u=0$, $\beta_0=1$ and $3\beta
_{\ell ^{\star }}>1)$ it is not possible to identify a single dominating
component; indeed, all the chaotic projections contribute with a variance of
the same rate $T,$ and the Gaussian limiting behaviour will follow from a
Breuer-Major type argument \cite[\S 5.3, \S 7]{noupebook}. 

\subsection{Discussion}

We can further summarize our results as follows:
\begin{center}
    \begin{tabular}{ | p{3cm} | p{3cm} | p{3cm} | p{3cm} |}
    \hline
                & $u\ne0$ & $u=0$ & asymptotic $\qquad$ distribution \\ \hline
   first chaos $\qquad$ dominates if & ${\beta_0<\min(2\beta_{\ell^\star},1)}$ $(\Var\approx T^{2-\beta_0})$ & ${\beta_0<\min(3\beta_{\ell^\star},1)}$ $(\Var\approx T^{2-\beta_0})$ & Gaussian \\ \hline
    second chaos  $\qquad$  dominates if & ${2\beta_{\ell^\star}<\min(\beta_0,1)}$ $(\Var\approx T^{2-2\beta_{\ell^\star}})$ & never  {(it vanishes)} & non-Gaussian (composite $\qquad$ Rosenblatt 2)\\ \hline
    third chaos  $\qquad$  dominates if & never & ${3\beta_{\ell^\star}<\min(\beta_0,1)}$ $(\Var\approx T^{2-3\beta_{\ell^\star}})$ & non-Gaussian \\ \hline
    all chaoses  $\qquad$  dominate if & $\beta_0=1$, $2\beta_{\ell^\star}>1$ $(\Var\approx T)$ & $\beta_0=1$, $3\beta_{\ell^\star}>1$ $(\Var\approx  T)$ & Gaussian \\
    \hline
    \end{tabular}
\end{center}
These findings should be compared with a rapidly growing literature devoted
to the investigation of geometric functionals over spherical random fields
in a different regime; in particular, a number of papers (see e.g.~\cite{CM2018,  MR2015,  MRW,  MW2014, Wig, Todino1})
have considered the high-frequency behaviour (e.g., when the eigenvalues
diverge) for spherical random eigenfunctions with no form of temporal
dependence. The results we exhibited here have some analogies, but also
important differences, with this stream of literature. In particular
\begin{itemize}
\item  for the excursion area of random spherical harmonics at $u\neq 0$ \cite{MW2014, MR2015} it is
indeed the case that the high-energy behaviour is dominated by the
second-order chaotic projection, whose asymptotic distribution is, however,
Gaussian. The same asymptotic behaviour occurs for other geometric
functionals, such as the boundary length of excursion sets and their
Euler-Poincar\'{e} characteristic, see \cite{Rossi2018, CM2018};
\item  for $u=0,$ the limiting variance is always of smaller-order, and
asymptotic Gaussianity holds \cite{MW2014, MRW}.
\end{itemize}
These differences can be explained as follows. Because in the case of
high-frequency asymptotics one deals with sequences of eigenspaces of growing dimensions,
the second chaotic components correspond to a sum of a growing number of 
i.i.d. coefficients, whence a standard Central Limit Theorem holds. In our case
here, the dimension of the sum of the eigenspaces which correspond to the \emph{strongest
memory} does not diverge in general, and hence asymptotic Gaussianity need
not to hold. Moreover, in the case of high-frequency asymptotics the
\emph{linear} projection term $a_{00}$ is dropped by
construction: on the contrary, for the random fields we investigate here
this term can be dominant for instance when $\beta _{0}<\min (2\beta _{\ell ^{\star }},1),$
in which case Gaussianity follows trivially. 

As far as Berry's cancellation
is concerned, this can occur in the present circumstances only when $%
H_{2}(Z(\cdot,\cdot ))$ exhibits long memory behaviour, i.e., non-integrable temporal
autocovariance: this is indeed the case for $2\beta _{\ell ^{\star }}<1.$ If
these condition is not met, all chaotic components have integrable temporal
autocovariance, none of them dominates and a Central Limit Theorem is
established by means of a Breuer-Major Theorem. Note that the presence of
long memory behaviour in the field $Z$ is a necessary, but not sufficient condition for the covariance of $H_{2}(Z(\cdot,\cdot ))$ to be
non-integrable.

As a final analogy, a remarkable feature of high-frequency asymptotics for
random eigenfunctions is the fact that geometric functionals turn out to be
asymptotically fully correlated over different levels, and even among
themselves, see \cite{CM19} and the references therein. It is then of interest to investigate if
similar features appear in the present framework. We present here a small
result that highlights this point.

\begin{proposition}\label{mono}
Assume that $u\ne 0$, $2\beta _{\ell^\star}<\min(\beta_0,1)$ and that there exists a unique 
\begin{equation*}
\ell ^{\star }:=\arg\min_{\ell \in \widetilde{\mathbb{N}},\ell\ge 1}\beta _{\ell }\,,
\end{equation*}%
then, as $T\rightarrow \infty $, 
\begin{equation*}
\Corr\left( \mathcal{M}_{T}(u),m_{T;\ell^{\star }}(u)\right)
\rightarrow 1\,,
\end{equation*}%
where 
\begin{equation*}
m_{T;\ell^{\star }}(u):=\frac{u}{2\sigma _{\ell^\star}}\phi
\left( \frac{u}{\sigma _{\ell ^{\star }}}\right) \int_{\mathbb{S}%
^{2}}\int_{0}^{T}H_{2}\left( \frac{Z_{\ell ^{\star }}(x,t)}{\sigma _{\ell
^{\star }}}\right) dx\,dt\,.
\end{equation*}
\end{proposition}

\begin{remark}
Note that, if we introduce the process 
\begin{equation*}\label{MTell}
\mathcal{M}_{T;\ell }(u)=\int_{0}^{T} \left ( \mathcal{A}_{u;\ell }(t) - \mathbb E[\mathcal{A}_{u;\ell }(t)] \right )\,dt\text{ ,}
\end{equation*}%
where%
\begin{equation*}
\mathcal{A}_{u;\ell }(t):=\int_{\mathbb{S}^{2}}\mathbf{1}_{Z_{\ell
}(x,t)\geq u}\,dx\,,
\end{equation*}%
(cf.~\eqref{eqZell} and \eqref{Au}) then $m_{T;\ell^{\star }}(u)=\mathcal{M}_{T;\ell^{\star }}(u)[2]$, the
second order chaotic component of the functional of the monochromatic field $Z_{\ell^\star}$.
\end{remark}

\section{Stroock-Varadhan
decompositions}\label{sec-Wiener}

The first tool that is needed in order to establish our asymptotic results is the derivation 
of the analytic form for the chaotic expansion \eqref{exp_s} of the empirical measure. The
result is very close to analogous findings given by \cite{DehTaq, MW2014}. For a complete discussion on Wiener chaos 
and related topics see e.g.
\cite[\S 2.2]{noupebook}.

\begin{lemma}\label{lem_chaos}
For every $T>0$ we have that 
\begin{equation}  \label{CE_M}
\mathcal{M}_T(u)= \sum_{q\ge 1} \frac{J_q(u)}{q!}\int_{[0,T]}\int_{\mathbb{S}^2}
H_q\left(Z(x,t)\right)\,dxdt
\end{equation}
where $J_q(u)=H_{q-1}(u)\phi(u)$, $\phi$ still being the density function of a
standard Gaussian random variable and $H_q$ the Hermite polynomial \eqref{Herm} of
order $q$. The convergence of the series in \eqref{CE_M} is in the $L^2(\Omega)$-sense. 
\end{lemma}

\begin{proof}
 Let $Z\sim \mathcal N(0,1)$, then 
\begin{equation*}
\mathbf{1}_{Z\ge u}=\sum_{q=0}^{\infty }\frac{J_{q}(u)}{q!}%
H_{q}(Z)\,,
\end{equation*}%
where the right hand side converges in the $L^{2}(\Omega)$-sense and
the coefficients $J_{q}(u)$ are given by 
\begin{flalign*}
&J_{q}(u) := \mathbb E[\mathbf{1}_{Z\ge u} H_q(Z)]
=\int_{\mathbb{R}}\mathbf{1}_{x\ge u}\,H_{q}(x)\phi (x)\,dx=(-1)^q \int_{u}^{+\infty}\frac{d^q}{dx^q} \phi(x)\,dx=H_{q-1}(u)\phi(u)
\end{flalign*}
(note that for fixed $x\in \mathbb S^2,t\in \mathbb R$, $Z(x,t)$ is standard Gaussian). Now consider the sequence of random variables 
$$
\left \lbrace \sum_{q=0}^{Q} \frac{J_{q}(u)}{q!} 
\int_{[0,T]}\int_{\mathbb S^2} H_{q}(Z(x,t))\,dxdt, \ Q\ge 1 \right \rbrace;
$$
let us prove that 
\begin{equation}\label{approxQ}
\sum_{q=0}^{Q} \frac{J_{q}(u)}{q!} 
\int_{[0,T]}\int_{\mathbb S^2} H_{q}(Z(x,t))\,dxdt \stackrel{Q\to +\infty}{\longrightarrow} \mathcal M_T(u)
\end{equation} 
in the $L^2(\Omega)$-sense. We have, thanks to Jensen inequality and Fubini-Tonelli Theorem,
\begin{eqnarray*}
&&\mathbb E\left [ \left ( M_T(u) -  \sum_{q=0}^{Q} \frac{J_{q}(u)}{q!} 
\int_{[0,T]}\int_{\mathbb S^2} H_{q}(Z(x,t))\,dxdt   \right )^2  \right ] \cr
&&= \mathbb E \left [ \left ( \int_{[0,T]}\int_{\mathbb S^2} \left ( \mathbf{1}_{Z(x,t)\ge u} -   \sum_{q=0}^{Q} \frac{J_{q}(u)}{q!}  H_{q}(Z(x,t))\right)dxdt     \right )^2       \right ] \cr
&&\le 4\pi T  \int_{[0,T]} \int_{\mathbb S^2} \mathbb E\left [ \left (  \mathbf{1}_{Z(x,t)\ge u} -   \sum_{q=0}^{Q} \frac{J_{q}(u)}{q!}  H_{q}(Z(x,t))     \right )^2\right ]\,dxdt  \cr
&&= (4\pi)^2 T^2  \mathbb E\left [ \left (  \mathbf{1}_{Z\ge u} -   \sum_{q=0}^{Q} \frac{J_{q}(u)}{q!}  H_{q}(Z)     \right )^2\right ]\rightarrow_{Q\to +\infty} 0,
\end{eqnarray*} 
hence \eqref{approxQ} holds and the proof is concluded. 
\end{proof}

Thanks to orthogonality of the chaotic components, from Lemma \ref{lem_chaos} we get 
\begin{eqnarray}  \label{var_M}
&&\Var\left(\mathcal{M}_T(u)\right)=\sum_{q=1}^\infty\Var\left(%
\mathcal{M}_T(u)[q]\right) \cr
&&=\sum_{q=1}^\infty \, \frac{J_q(u)^2}{q!}\, \int_{[0,T]^2} 
\int_{\mathbb{S}^2\times\mathbb{S}^2} \Gamma( \langle x,y
\rangle,t-s)^q\, dx dy dt ds,
\end{eqnarray}
where $\Gamma$ is the covariance function in \eqref{Gamma}.

\subsection{First order chaotic projections}

In this subsection we investigate the variance behavior of the first chaotic component  \eqref{eq1-discussion}. 
\begin{lemma}\label{lem1}
We have, as $T\to +\infty$,
$$
\lim_{T\to\infty} \frac{\Var(\mathcal M_T(u)[1])}{T^{2-\beta_0}} = \frac{2\phi(u)^2 C_0(0)}{(1-\beta_0 )(2-\beta_0 )},\quad \text{if } \beta_0\in (0,1)$$
and
$$
\lim_{T\to\infty}\frac{\Var(\mathcal M_T(u)[1])}{T} = \phi(u)^2 \int_{\mathbb R} C_0(\tau)\,d\tau,\quad \text{if }\beta_0=1.
$$
\end{lemma} 
Recall that Condition \ref{basic2} ensures that $C_0(0)>0$ and that for $\beta_0=1$ 
$$
\int_{\mathbb R} C_0(\tau)\,d\tau \in (0,+\infty)\,,
$$
hence Lemma \ref{lem1} gives the exact rate for the variance, the limiting constants being strictly positive.
From \eqref{eq1-discussion} we can write 
\begin{equation}\label{var1}
\Var(\mathcal M_T(u)[1]) = \phi(u)^2 \int_{[0,T]^2} C_0(t-s)\,dtds.
\end{equation}
\begin{remark}\label{remComput}
We will often make use of the following standard computation, that we report in this remark and that are taken for granted in the rest of the article. Making the change of variable $\tau=t-s$ 
for the double integral on the right hand side of \eqref{var1}, one has 
\begin{eqnarray*}
&&\int_{[0,T]^2}C_0(t-s)\,dt\,ds=\int_0^T ds \int_{-s}^{T-s} C_0(\tau)\,d\tau\\
&&=\int_0^T ds \int_{-T}^{T}\mathbf{1}_{[{-s},{T-s}]}(\tau) \,\,C_0(\tau)\,d\tau=\int_{-T}^{T} \,C_0(\tau)\,d\tau  \int_0^T \mathbf{1}_{[{-\tau},{T-\tau}]}(s) ds \\
&&=\int_{-T}^{0} \,C_0(\tau)\,d\tau  \int_0^T \mathbf{1}_{[{-\tau},{T}]}(s) ds +\int_{0}^{T} \,C_0(\tau)\,d\tau  \int_0^T \mathbf{1}_{[{0},{T-\tau}]}(s) ds \\
&&=\int_{-T}^{0} (T+\tau)\,C_0(\tau)\,d\tau+\int_{0}^{T}(T-\tau) \,C_0(\tau)\,d\tau  \\
&&=2T\,\int_{0}^{T} \left (1- \frac{\tau}{T} \right ) \,C_0(\tau)\,d\tau = T\int_{-T}^T  \left (1- \frac{|\tau|}{T} \right ) \,C_0(\tau)\,d\tau\,. 
\end{eqnarray*}
\end{remark}
It is now easy to investigate the asymptotic behavior, as $T\to +\infty$, of the variance of the first order chaotic component. 

\begin{proof}[Proof of Lemma \ref{lem1}]
From Remark \ref{remComput} and \eqref{var1} we can write  
\begin{eqnarray}\label{aus1}
\Var(\mathcal M_T(u)[1]) = 2T\phi(u)^2\int_0^T \left ( 1- \frac{\tau}{T} \right ) \,C_0(\tau)\,d\tau\,.
\end{eqnarray}
If $\beta_0=1$, recalling from Condition \ref{basic2} that in this case the covariance $C_0$ is integrable on $\mathbb R$, we immediately have (thanks to Dominated convergence Theorem) the exact asymptotic behavior of the variance 
$$
\lim_{T\to\infty} \frac{\Var(\mathcal M_T(u)[1])}{T} = 2\phi(u)^2\int_0^{+\infty} C_0(\tau)\,d\tau = \phi(u)^2\int_{\mathbb R} C_0(\tau)\,d\tau\,.
$$
Now assume $\beta_0<1$. Let $\varepsilon>0$, thanks to Condition \ref{basic2}, there exists $M>0$ such that, for $\tau>M$,
\begin{equation}\label{sup}
\sup_{\ell\in \widetilde{\mathbb N}}\left | \frac{G_\ell(\tau)}{C_\ell(0)}-1 \right |<\varepsilon \,,
\end{equation}
 and we can write (from \eqref{aus1})
\begin{eqnarray}\label{var1aus}
\Var(\mathcal M_T(u)[1]) &=&2T\phi(u)^2\int_0^M \left ( 1- \frac{\tau}{T} \right ) \,C_0(\tau)\,d\tau +  2T\phi(u)^2\int_M^T \left ( 1- \frac{\tau}{T} \right ) \,C_0(\tau)\,d\tau\cr
&=& O(1) + 2T\phi(u)^2\int_0^M C_0(\tau)\,d\tau + 2T\phi(u)^2\int_M^T \left ( 1- \frac{\tau}{T} \right ) \,C_0(\tau)\,d\tau\,.
\end{eqnarray} 
Consider the last integral on the right hand side of \eqref{var1aus} and write
\begin{eqnarray*}
\frac{1}{T^{1-\beta_0}}\int_M^T \left ( 1- \frac{\tau}{T} \right ) \,C_0(\tau)\,d\tau &=& \frac{C_0(0)}{T^{1-\beta_0}}\int_M^T \left ( 1- \frac{\tau}{T} \right ) (1+\tau)^{-\beta_0}\,d\tau\\
&& + \frac{C_0(0)}{T^{1-\beta_0}}\int_M^T \left ( 1- \frac{\tau}{T} \right ) \left(\frac{G_0(\tau)}{C_0(0)}-1\right)(1+\tau)^{-\beta_0}\,d\tau\,. 
\end{eqnarray*} 
We have 
\begin{equation}\label{aus2}
\lim_{T\to\infty}  \frac{C_0(0)}{T^{1-\beta_0}}\int_M^T \left ( 1- \frac{\tau}{T} \right ) (1+\tau)^{-\beta_0}\,d\tau = \frac{C_0(0)}{(1-\beta_0 )(2-\beta_0 )}
\end{equation}
and 
\begin{equation}\label{aus3}
\lim_{T\to\infty}  \frac{C_0(0)}{T^{1-\beta_0}}\int_M^T \left ( 1- \frac{\tau}{T} \right ) \left(\frac{G_0(\tau)}{C_0(0)}-1\right)(1+\tau)^{-\beta_0}\,d\tau = 0\,.
\end{equation}
The proof of \eqref{aus2} is straightforward; recall that by assumption $C_0(0) >0$. It remains to prove \eqref{aus3}. For $T>M$
\begin{eqnarray*}
&& \frac{C_0(0)}{T^{1-\beta_0}}\int_M^T \left ( 1- \frac{\tau}{T} \right ) \left|\frac{G_0(\tau)}{C_0(0)}-1\right |(1+\tau)^{-\beta_0}\,d\tau \le  \varepsilon \frac{C_0(0)}{T^{1-\beta_0}}\int_M^T  (1+\tau)^{-\beta_0}\,d\tau\\
 && =  \varepsilon \frac{C_0(0)}{1-\beta_0} \left ( \left ( 1 + \frac{1}{T}   \right )^{1-\beta_0} -    \left (\frac{M+1}{T}   \right )^{1-\beta_0}      \right ) \le \varepsilon \frac{C_0(0)}{1-\beta_0}\left(1+\frac 1T  \right)^{1-\beta_0}.
   \end{eqnarray*} 
Hence 
\begin{eqnarray*}
\limsup_{T\to +\infty} \left | \frac{C_0(0)}{T^{1-\beta_0}}\int_M^T \left ( 1- \frac{\tau}{T} \right ) \left (\frac{G_0(\tau)}{C_0(0)}-1\right)(1+\tau)^{-\beta_0}\,d\tau \right |\le  \varepsilon \frac{C_0(0)}{1-\beta_0}
\end{eqnarray*}
and the result follows, $\varepsilon$ being arbitrary.  
Plugging \eqref{aus2} and \eqref{aus3} into \eqref{var1aus} we find 
$$
\lim_{T\to\infty} \frac{\Var(\mathcal M_T(u)[1])}{T^{2-\beta_0}} = \frac{2\phi(u)^2 C_0(0)}{(1-\beta_0 )(2-\beta_0 )},\quad \beta_0\in (0,1)
$$
that concludes the proof. 
\end{proof}

\subsection{Second order chaotic projections}

Our next step is a careful analysis for the variance of the
second order chaotic component $\mathcal M_T(u)[2]$ (\eqref{exp_s} and \eqref{CE_M}, which will play a dominating role in most long memory scenarios (see \S \ref{sec_lm}). 
For $q=2$ we have 
\begin{equation*}
\Var\left( \mathcal{M}_{T}(u)[2]\right) =\frac{u^{2}\phi (u)^{2}}{2}%
\int_{[0,T]^{2}}\int_{\mathbb{S}^{2}\times \mathbb{S}^{2}}\Gamma (\langle
x,y\rangle ,t-s)^{2}\,dxdydtds\,.
\end{equation*}%
Now, thanks to \eqref{eq2} and \eqref{eqConv}, 
\begin{eqnarray*}
&&\int_{[0,T]^{2}}\int_{\mathbb{S}^{2}\times \mathbb{S}^{2}}\Gamma (\langle
x,y\rangle ,t-s)^{2}\,dx\,dy\,dt\,ds \\
&=&\int_{[0,T]^{2}}\int_{\mathbb{S}^{2}\times \mathbb{S}^{2}}\left(
\sum_{\ell =0}^{\infty }\,C_{\ell }(t-s)\,\frac{(2\ell +1)}{4\pi }\,P_{\ell
}(\langle x,y\rangle )\right) ^{2}\,dx\,dy\,dt\,ds \\
&=&\sum_{\ell _{1},\ell _{2}=0}^{\infty }\int_{[0,T]^{2}} C_{\ell
_{1}}(t-s)\,C_{\ell _{2}}(t-s)\,\frac{(2\ell _{1}+1)(2\ell _{2}+1)}{(4\pi
)^{2}}\times \cr
&&\times \int_{\mathbb{S}^{2}\times \mathbb{S}^{2}}P_{\ell _{1}}(\langle
x,y\rangle )\,P_{\ell _{2}}(\langle x,y\rangle )\,dx\,dy\,dt\,ds \\
&=& \sum_{\ell _{1},\ell _{2}=0}^{\infty }\int_{[0,T]^{2}} C_{\ell
_{1}}(t-s)\,C_{\ell _{2}}(t-s)\,\frac{(2\ell _{1}+1)(2\ell _{2}+1)}{(4\pi
)^{2}}   \frac{(4\pi)^2}{2\ell_1+1}\mathbf{1}_{\ell_1=\ell_2}\,dtds \\
&=&\sum_{\ell_1 =0}^{\infty }(2\ell_1 +1)\int_{[0,T]^{2}}\,C_{\ell_1
}(t-s)^{2}\,dt\,ds.
\end{eqnarray*}
Hence 
\begin{eqnarray*}
\Var\left( \mathcal{M}_{T}(u)[2]\right) &=&\frac{u^{2}\phi (u)^{2}}{2}%
\sum_{\ell =0}^{\infty }(2\ell +1)\int_{[0,T]^{2}}\,C_{\ell
}(t-s)^{2}\,dt\,ds \,.
\end{eqnarray*}
The next result is of fundamental importance for the study of the asymptotic behavior of $\Var\left( \mathcal{M}_{T}(u)[2]\right)$, its proof will be given in \S \ref{secVar2}.

\begin{lemma}\label{lemma-var2nd}
Fix $\ell \in \widetilde{ \mathbb{N}}$. If $2\beta _{\ell}<1$, then%
\begin{equation*}
\lim_{T\rightarrow \infty }\frac{1}{T^{2-2\beta _{\ell }}}%
\int_{[0,T]^2}C_{\ell }^{2}(t-s)dtds=\frac{C_{\ell }(0)^{2}}{(1-\beta_{\ell})(1-2\beta_{\ell})}\text{ .}
\end{equation*}
If $2\beta_\ell =1$, then 
\begin{equation*}
\lim_{T\rightarrow \infty }\frac{1}{T\log T}%
\int_{[0,T]^2}C_{\ell }^{2}(t-s)dtds=2 C_{\ell }(0)^{2}. 
\end{equation*} 
If $2\beta _{\ell}>1$, then
\begin{equation*}
\lim_{T\rightarrow \infty }\frac{1}{T}%
\int_{[0,T]^2}C_{\ell }^{2}(t-s)dtds=\int_{\mathbb R}C_\ell(\tau)^2\, d\tau\,.
\end{equation*}
\end{lemma}
Let us write 
\begin{flalign}
&\Var\left( \mathcal{M}_{T}(u)[2]\right) =\frac{u^{2}\phi (u)^{2}}{2}\int_{[0,T]^{2}} C_{0
}(t-s)^{2} dt\,ds+\frac{u^{2}\phi (u)^{2}}{2}%
\sum_{\ell =1}^{\infty }(2\ell +1)\int_{[0,T]^{2}} C_{\ell
}(t-s)^{2} dt\,ds \,. \label{eq2-var2nd} \qquad
\end{flalign}
Now recall the definition of $\beta_{\ell^\star}$ in \eqref{def_star}. 

\begin{proposition}\label{prop-var2nd}
\label{var2ndchaos} Assume $u\ne 0$. For $2\beta _{\ell ^{\star }}<1$ and $\beta_{\ell^\star}\le \beta_0$, we have that
\begin{equation}\label{th21}
\lim_{T\rightarrow \infty }\frac{\Var\left( \mathcal{M}
_{T}(u)[2]\right) }{T^{2-2\beta _{\ell ^{\star }}}}=\frac{u^{2}\phi (u)^{2}}{2(1-2\beta _{\ell ^{\star }})(1-\beta _{\ell ^{\star }})}\sum_{\ell \in \mathcal{I}^{\star }}(2\ell +1)\,C_{\ell}(0)^2\,;
\end{equation}
for $2\beta _{\ell ^{\star }}<1$ and $\beta_0 < \beta_{\ell^\star}$,
\begin{equation*}
\lim_{T\rightarrow \infty }\frac{\Var\left( \mathcal{M}%
_{T}(u)[2]\right) }{T^{2-2\beta _0}}=\frac{u^{2}\phi (u)^{2}}{2(1-2\beta _{0})(1-\beta _{0})}C_{0}(0)^2\,;
\end{equation*}
for $2\beta_{\ell^\star}= 1$ and $\beta_{\ell^\star}\le \beta_0$,
\begin{equation}\label{2log}
\lim_{T\rightarrow \infty }\frac{\Var\left( \mathcal{M}%
_{T}(u)[2]\right) }{T\log T}=u^{2}\phi (u)^{2}
\sum_{\ell \in \mathcal{I}^{\star }}(2\ell +1)\,C_{\ell}(0)^2\,;
\end{equation}
for $\beta_{\ell^\star}= \frac12$ and $\beta_0 <\beta_{\ell^\star}$, 
\begin{equation*}
\lim_{T\rightarrow \infty }\frac{\Var\left( \mathcal{M}%
_{T}(u)[2]\right) }{T^{2-2\beta_0}}=\frac{u^{2}\phi (u)^{2}}{2(1-2\beta _{0})(1-\beta _{0})}C_{0}(0)^2\,;
\end{equation*}
 for $2\beta _{\ell ^{\star }}>1$ and $2\beta_0 > 1$,
\begin{equation}\label{2T}
\lim_{T\rightarrow \infty }\frac{\Var\left( \mathcal{M}%
_{T}(u)[2]\right) }{T\,}=\frac{u^{2}\phi (u)^{2}}{2}\sum_{\ell =0}^{\infty }(2\ell +1)\int_{\mathbb R}C_\ell(\tau)^2\, d\tau\,;
\end{equation}
for  $2\beta _{\ell ^{\star }}>1$ and $2\beta_0 = 1$,
\begin{equation*}
\lim_{T\rightarrow \infty }\frac{\Var\left( \mathcal{M}%
_{T}(u)[2]\right) }{T\log T}=u^{2}\phi (u)^{2}\int_{\mathbb R}C_0(\tau)^2\,;
\end{equation*}
finally, for $2\beta _{\ell ^{\star }}>1$ and $2\beta_0 < 1$,
\begin{equation*}
\lim_{T\rightarrow \infty }\frac{\Var\left( \mathcal{M}%
_{T}(u)[2]\right) }{T^{2-2\beta _0}}=\frac{u^{2}\phi (u)^{2}}{2(1-2\beta _{0})(1-\beta _{0})}C_{0}(0)^2\,.
\end{equation*}
\end{proposition}
Recall that by assumption $C_0(0)>0$, and for $\ell\in \mathcal I^\star$ we have $C_{\ell}(0)>0$ (see \eqref{def_star}); as a consequence, Proposition \ref{prop-var2nd} gives the exact rate for the variance, the limiting constants being strictly positive.

\begin{remark}
In words, for $\beta_{\ell^\star}\le \beta_0$, when $2\beta _{\ell ^{\star }}<1$ (resp.~$2\beta _{\ell ^{\star }}=1$), we
have a form of long-range dependence and the second order chaotic component
of the functional $\mathcal{M}_{T}(u)$ is dominated by a subset of the
multipoles; the variance scales as order $T^{2-2\beta _{\ell ^{\star }}}$ (resp.~$T\log T$). On
the contrary, when $2\beta _{\ell ^{\star }}>1$, a
form of short-range dependence holds and all frequencies contribute with
variance terms of order $T$.
\end{remark}
In order to prove Proposition \ref{prop-var2nd} we will also need the following technical results. The proofs of Lemma \ref{tec2} and \ref{tec2_1} are given in the Appendix \S \ref{app}, the proofs of Lemma \ref{tec2_2} and Lemma \ref{tec2_3}
are indeed similar and we omit the details for brevity. 
\begin{lemma}\label{tec2}
Let $\varepsilon, M>0$ be as in \eqref{sup}.  
For $\ell$ such that $\beta_\ell =1$ and $T> \max(1, M)$
\begin{eqnarray}\label{beta1}
\frac{1}{T}\int_{[0,T]^{2}}\,C_{\ell}(t-s)^{2}\,dtds \le 2C_\ell(0)^2 \left ( M + 2\frac{(\varepsilon+1)^2}{\alpha -1} \right),
\end{eqnarray}
where $\alpha \ge 2$ comes from the definition in \eqref{gbeta}. 
\end{lemma}
\begin{lemma}\label{tec2_1}
Let $\varepsilon, M>0$ be as in \eqref{sup} and $2\beta_{\ell^\star} <1$.
\begin{itemize} 
\item For $\ell\in \mathcal I^\star$, $\beta_\ell<1$ and $T> \max(1, M)$ 
\begin{eqnarray}\label{RHS}
\frac{1}{T^{2-2\beta _{\ell ^{\star }}}}\int_{[0,T]^{2}}\,C_{\ell}(t-s)^{2}\,dtds \le 2C_\ell(0)^2 \left ( M +  \frac{(\varepsilon +1)^2}{-2\beta_{\ell^\star}+1}\left ( 1 + \frac1M    \right )^{1-2\beta_{\ell^\star}}  \right ).
\end{eqnarray} 
\item 
Let $m(\beta_{\ell^\star}) := \max_{x>0}\frac{\log (1+x)}{x^{1-2\beta_{\ell^\star}}}$ and $T_m$ the corresponding $\arg\max$. 
For $\ell\notin \mathcal I^\star$, $\ell\ge 1$, $\beta_\ell <1$ and $T>\max(1,M, T_m)$ we have 
\begin{flalign}\label{RHS2}
\frac{1}{T^{2-2\beta _{\ell ^{\star }}}}\int_{[0,T]^{2}}\,C_{\ell}(t-s)^{2}\,dtds &\le 2C_\ell(0)^2 \Big ( M + 2m(\beta_{\ell^\star}) (\varepsilon + 1)^2 \mathbf{1}_{2\beta_{\ell^{\star\star}}=1}\notag \\
& \quad+ (\varepsilon+1)^2 \frac{1}{-2\beta_{\ell^{\star\star}}+1}\left (1 + \frac1M\right )^{-2\beta_{\ell^{\star\star}}+1}\mathbf{1}_{2\beta_{\ell^{\star\star}}<  1} \notag\\
&\quad + (\varepsilon +1)^2\frac{1}{2\beta _{\ell ^{\star\star }}-1}\left (\frac{1}{1+M}    \right )^{2\beta_{\ell^{\star\star}}-1}\mathbf{1}_{2\beta_{\ell^{\star\star}}>  1}\Big ).
\end{flalign} 
\end{itemize}
\end{lemma} 
\begin{lemma}\label{tec2_2}
Let $\varepsilon, M>0$ be as in \eqref{sup} and $2\beta_{\ell^\star} =1$.
\begin{itemize} 
\item For $\ell\in \mathcal I^\star$, $\beta_\ell<1$ and $T> \max(1, M,e)$ 
\begin{eqnarray*}
\frac{1}{T\log T}\int_{[0,T]^{2}}\,C_{\ell}(t-s)^{2}\,dtds \le 2C_\ell(0)^2 \left ( M + \log(e+1)  \right ).
\end{eqnarray*} 
\item 
For $\ell\notin \mathcal I^\star$, $\ell\ge 1$, $\beta_\ell<1$ and $T> \max(1, M)$ 
\begin{eqnarray*}
\frac{1}{T\log T}\int_{[0,T]^{2}}\,C_{\ell}(t-s)^{2}\,dtds \le 2C_\ell(0)^2 \left ( M + \int_{\mathbb R} (1 +|\tau|)^{-2\beta_{\ell^{\star\star}}})\,d\tau  \right ).
\end{eqnarray*} 

\end{itemize}
\end{lemma} 
\begin{lemma}\label{tec2_3}
Let $\varepsilon, M>0$ be as in \eqref{sup}. If $2\beta_{\ell^\star}>1$, for $\ell\ge 1$ and $\beta_\ell<1$ we have  
\begin{equation*}
\frac{1}{T}\int_{[0,T]^2} C_\ell(t-s)^2\,dtds \le 2C_{\ell}(0) \left ( M + (\varepsilon+1)^2 \int_{\mathbb R} (1 +|\tau|)^{-2\beta_{\ell^\star}}\,d\tau \right )
\end{equation*} 
whenever $T>\max(1,M)$.
\end{lemma} 
We are now in the position to prove Proposition \ref{prop-var2nd}. 

\begin{proof}[Proof of Proposition \ref{prop-var2nd}] 
Assume first that $2\beta _{\ell ^{\star }}<1$ and $\beta_{\ell^\star}\le \beta_0$. For the asymptotic behavior of the first term on the right hand side of \eqref{eq2-var2nd}) we refer to Lemma \ref{lemma-var2nd}:
\begin{eqnarray}\label{0}
\lim_{T\to\infty} \frac{\int_{[0,T]^2} C_0(t-s)^2\,dtds}{T^{2-2\beta_{\ell^\star}}} = \begin{cases}
0 \qquad &\text{if } \beta_{\ell^\star} < \beta_0,\\
\frac{C_{0}(0)^{2}}{(1-\beta_{\ell^\star})(1-2\beta_{\ell^\star})} \qquad &\text{if } \beta_{\ell^\star} = \beta_0\,.
\end{cases}
\end{eqnarray}
Now, since from \eqref{eqConv} we have
\begin{equation}\label{convSeries2}
\sum_{\ell=0}^{+\infty}  (2\ell +1) C_{\ell}(0)^{2}  < +\infty\,,
\end{equation}
thanks to Lemma \ref{tec2} and Lemma \ref{tec2_1} we can apply Dominate Convergence Theorem and then Lemma \ref{lemma-var2nd} to get 
\begin{eqnarray}
&&\lim_{T\rightarrow \infty }\sum_{\ell \in\mathcal{I}^\star}\frac{(2\ell +1)}{T^{2-2\beta _{\ell ^{\star }}}}\int_{[0,T]^{2}}\,C_{\ell}(t-s)^{2}\,dtds  \cr
&& = \sum_{\ell \in\mathcal{I}^\star} \lim_{T\rightarrow \infty } \frac{(2\ell +1)}{T^{2-2\beta _{\ell ^{\star }}}}\int_{[0,T]^{2}}C_{\ell}(t-s)^{2}\,dtds =  \frac{\sum_{\ell \in\mathcal{I}^\star} (2\ell +1) C_{\ell}(0)^{2}}{(1-\beta_{\ell^\star})(1-2\beta_{\ell^\star})}.\label{2aus1}
\end{eqnarray} 
Let us now prove that
\begin{equation}\label{notin}
\lim_{T\to\infty} \sum_{\ell \notin\mathcal{I}^\star, \ell\ge 1}\frac{(2\ell +1)}{T^{2-2\beta_{\ell^\star}}}\int_{[0,T]^{2}}\,C_{\ell}(t-s)^{2}\,dt\,ds  = 0\,.
\end{equation} 
Thanks again to Lemma \ref{tec2} and Lemma \ref{tec2_1}, since \eqref{convSeries2}) holds, we can apply Dominated Convergence Theorem  to obtain \eqref{notin}), 
more precisely we have to distinguish between the possible cases: from Lemma \ref{lemma-var2nd}\begin{eqnarray*} 
&&\lim_{T\to\infty}\sum_{\ell\ge 1:\beta_{\ell^\star}<\beta_\ell<\frac12} \frac{(2\ell +1)}{T^{2-2\beta_{\ell^\star}}}\int_{[0,T]^{2}}\,C_{\ell}(t-s)^{2}\,dt\,ds\\
&& =\sum_{\ell\ge 1:\beta_{\ell^\star}<\beta_\ell<\frac12} (2\ell +1) \lim_{T\rightarrow \infty }\underbrace{\frac{T^{2-2\beta _{\ell}}}{T^{2-2\beta _{\ell ^{\star }}}}}_{\rightarrow0}
\underbrace{\frac{1}{T^{2-2\beta _{\ell}}}\int_{[0,T]^{2}}\,C_{\ell}(t-s)^{2}\,dt\,ds}_{\rightarrow C_\ell(0)^2(1-\beta_\ell)^{-1}(1-2\beta_\ell)^{-1}} = 0;
\end{eqnarray*}
moreover 
\begin{eqnarray*}
&&\lim_{T\to\infty} \sum_{\ell\ge 1:\beta_\ell=\frac12}\frac{(2\ell +1)}{T^{2-2\beta_{\ell^\star}}}
\int_{[0,T]^{2}}\,C_{\ell}(t-s)^{2}\,dt\,ds\\
&& = \sum_{\ell\ge 1:\beta_\ell=\frac12} (2\ell +1)\lim_{T\to\infty}\underbrace{\frac{T\log T}{T^{2-2\beta_{\ell^\star}}}}_{\to 0}
\underbrace{\frac{1}{T\log T}\int_{[0,T]^{2}}\,C_{\ell}(t-s)^{2}\,dt\,ds}_{\to 2 C_{\ell }(0)^{2}}= 0\,;
\end{eqnarray*} 
and 
\begin{eqnarray*}  
&&\lim_{T\to\infty} \sum_{\ell\ge 1:\beta_\ell > \frac12} \frac{2\ell+1}{T^{2-2\beta_{\ell^\star}}}
\int_{[0,T]^{2}}\,C_{\ell}(t-s)^{2}\,dt\,ds \\
&&= \sum_{\ell\ge 1:\beta_\ell > \frac12}(2\ell+1) \lim_{T\to\infty} \underbrace{\frac{T}{T^{2-2\beta _{\ell ^{\star }}}}}_{\rightarrow0}
\underbrace{\frac 1 T\int_{[0,T]^{2}}\,C_{\ell}(t-s)^{2}\,dt\,ds}_{\rightarrow \int_{\mathbb R} C_\ell(\tau)^2\,d\tau}=0\,.
\end{eqnarray*}
Putting together \eqref{0}, \eqref{2aus1} and \eqref{notin} we immediately get \eqref{th21} in Proposition \ref{prop-var2nd}. 

Now assume $2\beta_0>1$ and $2\beta _{\ell ^{\star }}>1$. Then obviously $2\beta _{\ell}>1$ for each $\ell \in \widetilde{\mathbb{N}}$ and hence, 
 using Lemma \ref{tec2} and \ref{tec2_3} and Lemma \ref{lemma-var2nd} as before, we have
\begin{eqnarray*}
\lim_{T\rightarrow \infty }\frac{\Var\left( \mathcal{M}%
_{T}(u)[2]\right) }{T}&=&\lim_{T\rightarrow \infty }\frac{u^{2}\phi (u)^{2}}{2}\sum_{\ell =0}^{\infty }(2\ell +1)\frac1T \int_{[0,T]^{2}}\,C_{\ell}(t-s)^{2}\,dt\,ds\\
&=&\frac{u^{2}\phi (u)^{2}}{2}\sum_{\ell =0}^{\infty }(2\ell +1) \lim_{T\rightarrow \infty }2\int_{0}^{T}\left(1-\frac{\tau}{T}\right)C_{\ell }^{2}(\tau )d\tau\\
&=&\frac{u^{2}\phi (u)^{2}}{2}\sum_{\ell =0}^{\infty }(2\ell +1)\int_{-\infty}^{+\infty} C_\ell(\tau)^2\,d\tau\,,
\end{eqnarray*}
which is \eqref{2T}. 
Note that we automatically get 
$$
\sum_{\ell =0}^{\infty }(2\ell +1)\int_{-\infty}^{+\infty} C_\ell(\tau)^2\,d\tau < +\infty\,,
$$
and the proof is concluded, the remaining cases requiring analogous proofs. 
\end{proof}

\subsubsection{Proof of Lemma \ref{lemma-var2nd}}\label{secVar2}

\begin{proof} If $\beta_\ell \in \left (\frac12, 1 \right ]$ then  
from Remark \ref{remComput}, thanks to Dominated Convergence Theorem,
\begin{flalign*}
&\lim_{T\rightarrow \infty }\frac{1}{T}\int_{[0,T]^2}C_{\ell }^{2}(t-s)\,dtds = \lim_{T\rightarrow \infty}\int_{-T}^{T}\left(1-\frac{|\tau|}{T}\right)C_{\ell }(\tau)^2\,d\tau=\int_{\mathbb R}C_{\ell }(\tau)^2\,d\tau\,.
\end{flalign*}
Now assume that $2\beta _{\ell}<1$ and recall Condition \ref{basic2}, then as in Remark \ref{remComput} and the proof of Lemma \ref{lem1} we fix $\varepsilon>0$ and we know there exists $M>0$ such that, for $\tau>M$,
$$
\sup_{\ell}\left | \frac{G_\ell(\tau)}{C_\ell(0)}-1 \right |<\varepsilon
$$
(as in \eqref{sup}), so that
\begin{flalign}\label{aus_sec}
&\int_{[0,T]^2}C_{\ell }^{2}(t-s)dtds =2T\int_{0}^{T}\left(1-\frac{\tau}{T}\right)C_{\ell }^{2}(\tau )d\tau  \notag \\
&=2T\int_{0}^{M}\left(1-\frac{\tau}{T}\right)C_{\ell }^{2}(\tau )d\tau+2T\int_{M}^{T}\left(1-\frac{\tau}{T}\right)C_{\ell }^{2}(\tau )d\tau \text{ .}\notag \\
&= O(1) + 2T\int_{0}^{M}C_{\ell }^{2}(\tau )d\tau + 2TC_\ell(0)^2\int_{M}^{T}\left(1-\frac{\tau }{T}\right)\left (\frac{G_{\ell }(\tau)}{C_{\ell
}(0)}-1\right )^2(1+\tau)^{-2\beta _{\ell }} d\tau \notag \\
&\quad+ 4T\,C_\ell(0)^2\int_{M}^{T}\left(1-\frac{\tau }{T}\right)\left (\frac{G_{\ell }(\tau)}{C_{\ell
}(0)}-1\right )(1+\tau)^{-2\beta _{\ell }} d\tau\notag\\
&\quad+2TC_{\ell }(0)^{2}\int_{M}^{T}\left(1-\frac{\tau}{T}\right)(1+\tau)^{-2\beta _{\ell
}}d\tau \,.
\end{flalign}
For the second and the last summands of \eqref{aus_sec} it is straightforward to check that%
\begin{eqnarray*}
&&\lim_{T\to\infty} \frac{2T}{T^{2-2\beta_\ell}} \int_{0}^{M}C_{\ell }^{2}(\tau )d\tau = 0,\\
&&\lim_{T\rightarrow \infty }\frac{2TC_{\ell}(0)^{2}}{T^{2-2\beta _{\ell }}}\int_{M}^{T}\left(1-\frac{\tau}{T}\right)(1+\tau)^{-2\beta _{\ell }}d\tau =\frac{%
C_{\ell }(0)^{2}}{(1-\beta _{\ell })(1-2\beta _{\ell })}\text{ .}
\end{eqnarray*}
On the other hand, for the third and the fourth summands
\begin{eqnarray}\label{uniform0}
&&\lim_{T\rightarrow \infty }\frac{2TC_\ell(0)}{T^{2-2\beta _{\ell }}}\int_{M}^{T}\left(1-\frac{\tau }{T}\right)\left (\frac{G_{\ell }(\tau)}{C_{\ell
}(0)}-1\right )^2(1+\tau)^{-2\beta _{\ell }} d\tau   \cr
&&=\lim_{T\rightarrow \infty }\frac{4T\,C_\ell(0)}{T^{2-2\beta _{\ell }}}\int_{M}^{T}\left(1-\frac{\tau }{T}\right)\left (\frac{G_{\ell }(\tau)}{C_{\ell
}(0)}-1\right )(1+\tau)^{-2\beta _{\ell }} d\tau =0\,.
\end{eqnarray}%
Let us prove \eqref{uniform0}:  we have that 
\begin{eqnarray*}
&&\limsup_{T\to +\infty} \frac{2TC_\ell(0)}{T^{2-2\beta _{\ell }}}\int_{M}^{T}\left(1-\frac{\tau }{T}\right)\left (\frac{G_{\ell }(\tau)}{C_{\ell
}(0)}-1\right )^2(1+\tau)^{-2\beta _{\ell }} d\tau \\
&&\le \varepsilon^2 \limsup_{T\to +\infty} \frac{2C_\ell(0)}{T^{1-2\beta _{\ell }}}\int_{M}^{T}(1+\tau)^{-2\beta _{\ell }} d\tau \le \varepsilon^2 \frac{4C_\ell(0)}{1-2\beta_\ell}\,,
\end{eqnarray*}
and analogously that
\begin{eqnarray*}
&&\limsup_{T\to +\infty} \frac{4TC_\ell(0)}{T^{2-2\beta _{\ell }}}\int_{M}^{T}\left(1-\frac{\tau }{T}\right)\left | \frac{G_{\ell }(\tau)}{C_{\ell
}(0)}-1\right | (1+\tau)^{-2\beta _{\ell }} d\tau\le \varepsilon \frac{8C_\ell(0)}{1-2\beta_\ell}\,,
\end{eqnarray*}
and \eqref{uniform0} follows, $\varepsilon$ being arbitrary.
When $2\beta_\ell = 1$, then one can prove using the same arguments that
$$
\lim_{T\to\infty} \frac{1}{T\log T}\int_{[0,T]^2}C_{\ell }^{2}(t-s)dtds=2 C_{\ell}(0)^{2}
$$
and the proof of the lemma is concluded. \end{proof}

\subsection{Higher order chaotic projections}

In this subsection we want to investigate the behavior of higher order chaotic components. Let $q\ge 3$, from \eqref{CE_M}
we can write 
\begin{equation}\label{esprq}
\Var(\mathcal M_T(u)[q]) = \frac{\phi(u)^2H_{q-1}(u)^2}{q!} \int_{[0,T]^{2}}\int_{\mathbb{S}^{2}\times \mathbb{S}^{2}}\Gamma (\langle
x,y\rangle ,t-s)^{q}\,dx\,dy\,dt\,ds.
\end{equation}
Thanks to \eqref{eq2} we have that
\begin{eqnarray*}
&&\int_{[0,T]^{2}}\int_{\mathbb{S}^{2}\times \mathbb{S}^{2}}\Gamma (\langle
x,y\rangle ,t-s)^{q}\,dx\,dy\,dt\,ds \\
&=&\int_{[0,T]^{2}}\int_{\mathbb{S}^{2}\times \mathbb{S}^{2}}\left(
\sum_{\ell =0}^{\infty }\,C_{\ell }(t-s)\,\frac{(2\ell +1)}{4\pi }\,P_{\ell
}(\langle x,y\rangle )\right) ^{q}\,dx\,dy\,dt\,ds \\
&=&\sum_{\ell
_{1},\ell _{2},\dots ,\ell _{q}=0}^{\infty }\int_{[0,T]^{2}} \int_{\mathbb{S}^{2}\times \mathbb{S}^{2}} C_{\ell _{1}}(t-s)\,C_{\ell
_{2}}(t-s)\cdots C_{\ell _{q}}(t-s) \\
&&\times \frac{2\ell _{1}+1}{4\pi }\,P_{\ell _{1}}(\langle x,y\rangle )\,\frac{%
2\ell _{2}+1}{4\pi }P_{\ell _{2}}(\langle x,y\rangle )\cdots \frac{2\ell
_{q}+1}{4\pi }P_{\ell _{q}}(\langle x,y\rangle )\,dx\,dy\,dt\,ds\,.
\end{eqnarray*}%
Recall the addition formula for spherical harmonics \cite[(3.42)]{MaPeCUP}
$$
\frac{4\pi}{2\ell+1} \sum_{m=-\ell}^\ell Y_{\ell,m}(x) Y_{\ell,m}(y) = P_\ell(\langle x,y\rangle),\qquad x,y\in \mathbb S^2\,,
$$
and the definition of generalized Gaunt integral in \eqref{eq:gaunt} to write
\begin{eqnarray}\label{Gaunt+}
&&\int_{\mathbb S^2\times \mathbb S^2} \frac{2\ell _{1}+1}{4\pi }\,P_{\ell _{1}}(\langle x,y\rangle )\,\frac{%
2\ell _{2}+1}{4\pi }P_{\ell _{2}}(\langle x,y\rangle )\cdots \frac{2\ell
_{q}+1}{4\pi }P_{\ell _{q}}(\langle x,y\rangle )\,dxdy\cr
&&=\sum_{m_1=-\ell_1}^{\ell_1} \cdots \sum_{m_q =-\ell_q}^{\ell_q} \left ( \mathcal{G}_{\ell _{1}...\ell _{q}}^{m_1...m_q} \right )^2\,.
\end{eqnarray} 
Equivalently, 
\begin{eqnarray}\label{Gaunt2}
&&\int_{\mathbb S^2\times \mathbb S^2} \frac{2\ell _{1}+1}{4\pi }\,P_{\ell _{1}}(\langle x,y\rangle )\,\frac{%
2\ell _{2}+1}{4\pi }P_{\ell _{2}}(\langle x,y\rangle )\cdots \frac{2\ell
_{q}+1}{4\pi }P_{\ell _{q}}(\langle x,y\rangle )\,dxdy\cr
&&=4\pi \left ( \prod\limits_{i=1}^{q}\sqrt{\frac{2\ell _{i}+1}{4\pi }}\right )
\mathcal{G}_{\ell _{1}\ell _{2}...\ell _{q}}^{0\,...0}\,.
\end{eqnarray} 
In particular $\mathcal{G}_{\ell _{1}\dots \ell _{q}}^{0\dots 0}\ge 0$.
In order to check \eqref{Gaunt2} recall that
$$
Y_{\ell,0}(\theta_x,\varphi_x)=\sqrt{\frac{2\ell+1}{4\pi}}P_\ell(\cos \theta_x)\,,
$$
where $(\theta_x,\varphi_x)$ are the angular coordinates of the point $x\in\mathbb{S}^2$; then, letting $o$ be the \emph{north pole} of the sphere, we have 
\begin{eqnarray*}
&&\int_{\mathbb{S}^{2}\times \mathbb S^2}
\sqrt{\frac{2\ell _{1}+1}{4\pi }}P_{\ell _{1}}(\langle x,y\rangle )\cdots \sqrt{%
\frac{2\ell _{q}+1}{4\pi }}P_{\ell _{q}}(\langle x,y\rangle )dxdy\\
&&=4\pi \int_{\mathbb{S}^{2}}\sqrt{\frac{2\ell _{1}+1}{4\pi }}P_{\ell _{1}}(\langle x,o\rangle )\cdots \sqrt{%
\frac{2\ell _{q}+1}{4\pi }}P_{\ell _{q}}(\langle x,o\rangle )dx\\
&&=4\pi \int_{\mathbb{S}^{2}}\sqrt{\frac{2\ell _{1}+1}{4\pi }}P_{\ell _{1}}(\cos \theta_x )\cdots \sqrt{%
\frac{2\ell _{q}+1}{4\pi }}P_{\ell _{q}}(\cos \theta_x )dx=4\pi\, \mathcal{G}_{\ell _{1}\dots \ell _{q}}^{0\,...0}\,.
\end{eqnarray*}
As a consequence, from \eqref{esprq} we can write 
\begin{equation}\label{eq:varqth}
\Var(\mathcal{M}_T(u)[q]))=\frac{4\pi\,H_{q-1}(u)^2\phi(u)^2}{q!} \sum_{\ell
_{1},\dots ,\ell _{q}=0}^{\infty }k_{\ell _{1}\dots \ell
_{q}}(T)\left ( \prod\limits_{i=1}^{q}\sqrt{\frac{2\ell _{i}+1}{4\pi }}\right )
\mathcal{G}_{\ell _{1}\dots \ell _{q}}^{0\,...0}\,,
\end{equation}%
where%
\begin{equation}\label{defK}
k_{\ell _{1},\dots,\ell _{q}}(T):=\int_{[0,T]^{2}}C_{\ell _{1}}(t-s)\,C_{\ell _{2}}(t-s)\cdots C_{\ell_{q}}(t-s)dtds\,.
\end{equation}
Note that 
$$
k_{\ell_1, \dots, \ell_q}(T) = \mathbb E \left [ \left ( \int_{[0,T]} a_{\ell_1, 0}(t) \cdots a_{\ell_q,0}(t)\,dt   \right )^2 \right] .
$$
In order to study the asymptotic behavior, as $T\to +\infty$, of \eqref{eq:varqth} we will need the following result whose proof is given in \S \ref{secVarq}. 
\begin{lemma}\label{lemma-varqth}
Let $\ell_1,\dots,\ell_q$ be such that $\beta_{\ell_1}+\cdots+\beta_{\ell_q}<1$, then
\begin{equation}\label{eq1-lemma-varqth}
\lim_{T\rightarrow\infty}\frac{k_{\ell_1\dots\ell_q}(T)}{T^{2-(\beta_{\ell_1}+\cdots+\beta_{\ell_q})}}=\frac{C_{\ell _{1}}(0)\cdots C_{\ell_{q}}(0)}{(1-(\beta_{\ell_1}+\cdots+\beta_{\ell_q}))(2-(\beta_{\ell_1}+\cdots+\beta_{\ell_q}))}\,.
\end{equation}
and if $\beta_{\ell_1}+\cdots+\beta_{\ell_q}=1$
$$
\lim_{T\to\infty} \frac{k_{\ell_1\dots\ell_q}(T)}{T\log T} = 2 \,C_{\ell _{1}}(0)\cdots C_{\ell_{q}}(0)\,.
$$
On the contrary, let $\ell_1,\dots,\ell_q$ be such that $\beta_{\ell_1}+\cdots+\beta_{\ell_q}>1$, then 
\begin{equation}\label{eq2-lemma-varqth}
\lim_{T\to\infty} \frac{k_{\ell_1\dots\ell_q}(T)}{T}= \int_{-\infty}^{+\infty} C_{\ell _{1}}(\tau)\,C_{\ell _{2}}(\tau)\cdots C_{\ell_{q}}(\tau)\,d\tau\,.
\end{equation}
\end{lemma}
Recall \eqref{eq:varqth}. 
\begin{proposition}\label{prop-varqth}
Let $q\ge3$. If $q\beta_{\ell^\star} <1$ and $\beta_{\ell^\star}\le \beta_0$ then
\begin{flalign*}
&\lim_{T\rightarrow \infty }\frac{\Var\left(\mathcal{M}_T(u)[q]\right)}{T^{2-q\beta _{\ell^\star}}} =\frac{4\pi\,H_{q-1}(u)^2\phi(u)^2}{q!(1-q\beta_{\ell^\star})(2-q\beta_{\ell^\star})}\sum_{\ell
_{1},\ell _{2},\dots ,\ell _{q}\in \mathcal{I}^\star} \left ( \prod\limits_{i=1}^{q}\sqrt{\frac{2\ell _{i}+1}{4\pi }}C_{\ell_{i}}(0)\right )
\mathcal{G}_{\ell _{1}\dots \ell _{q}}^{0\,...0}\,;
\end{flalign*}
on the other hand, If $q\beta_{\ell^\star} <1$ and $\beta_{\ell^\star} > \beta_0$ then 
\begin{eqnarray*}
&&\lim_{T\rightarrow \infty }\frac{\Var\left(\mathcal{M}_T(u)[q]\right)}{T^{2-q\beta _0}} = \frac{H_{q-1}(u)^2\phi(u)^2}{(4\pi)^{q-2} q!} \frac{C_0(0)^q}{(1-q\beta_0)(2-q\beta_0)}\,.
\end{eqnarray*}
If $q\beta_{\ell^\star}=1$ and $\beta_{\ell^\star}\le \beta_0$ then 
$$
\lim_{T\to\infty} \frac{\Var\left(\mathcal{M}_T(u)[q]\right)}{T\log T} = \frac{8\pi\,H_{q-1}(u)^2\phi(u)^2}{q!}\sum_{\ell
_{1},\ell _{2},\dots ,\ell _{q}\in \mathcal{I}^\star }  \left ( \prod\limits_{i=1}^{q}\sqrt{\frac{2\ell _{i}+1}{4\pi }}C_{\ell_{i}}(0)\right )
\mathcal{G}_{\ell _{1} \dots \ell _{q}}^{0\,...0}\,;
$$
if $q\beta_{\ell^\star}=1$ and $\beta_{\ell^\star} > \beta_0$ then 
\begin{eqnarray*}
&&\lim_{T\rightarrow \infty }\frac{\Var\left(\mathcal{M}_T(u)[q]\right)}{T^{2-q\beta _0}} = \frac{H_{q-1}(u)^2\phi(u)^2}{(4\pi)^{q-2} q!} \frac{C_0(0)^q}{(1-q\beta_0)(2-q\beta_0)}\,.
\end{eqnarray*}
On the other hand, if $q\beta_{\ell^\star} >1$ and $q\beta_0 > 1$, then
$$
\lim_{T\to\infty}\frac{\Var\left(\mathcal{M}_T(u)[q]\right)}{T}=s^2_q\,,
$$
where 
$$
s^2_q :=  \frac{4\pi\,H_{q-1}(u)^2\phi(u)^2}{q!}\sum_{\ell
_{1},\ell _{2},\dots ,\ell _{q}=0}^{\infty }\mathcal{G}_{\ell _{1}...\ell _{q}}^{0\,...0} \int_{-\infty}^{+\infty} \left(\prod\limits_{i=1}^{q}\sqrt{\frac{2\ell _{i}+1}{4\pi }}C_{\ell _{i}}(\tau)\right )\,d\tau\,;
$$
moreover if $q\beta_{\ell^\star} >1$ and $q\beta_0 = 1$, then
\begin{eqnarray*}
&&\lim_{T\rightarrow \infty }\frac{\Var\left(\mathcal{M}_T(u)[q]\right)}{T\log T} = \frac{H_{q-1}(u)^2\phi(u)^2}{(4\pi)^{q-2} q!} 2C_0(0)^q;
\end{eqnarray*}
finally if $q\beta_{\ell^\star} >1$ and $q\beta_0 < 1$, then
\begin{eqnarray*}
&&\lim_{T\rightarrow \infty }\frac{\Var\left(\mathcal{M}_T(u)[q]\right)}{T^{2-q\beta_0}} =  \frac{H_{q-1}(u)^2\phi(u)^2}{(4\pi)^{q-2} q!} \frac{C_0(0)^q}{(1-q\beta_0)(2-q\beta_0)}\,.
\end{eqnarray*}
\end{proposition}
In order to prove Proposition \ref{prop-varqth} we will also need the following technical results, the proofs of Lemma \ref{lemmaqq} and Lemma \ref{tecq_1}  are postponed to the Appendix \S \ref{app}, 
the proofs of the remaining lemmas are very similar and we omit the details. 

\begin{lemma}\label{lemmaqq}
Let $\varepsilon, M>0$ be as in \eqref{sup}. 
If there is at least one index $j\in \lbrace 1,\dots,q\rbrace$ such that 
$\beta_{\ell_j}=1$ we have for $T>\max(1,M)$,
\begin{eqnarray*}
\frac{k_{\ell_1\dots \ell_q}(T)}{T} \le 2C_{\ell_1}(0)\dots C_{\ell_q}(0)\left ( M +  \frac{1}{q\min(\beta_0,\beta_{\ell^\star})} \left (  \frac{\varepsilon+1}{(1+M)^{\min(\beta_0,\beta_{\ell^\star})}}  \right)^q
\right ).
\end{eqnarray*}
\end{lemma}

\begin{lemma}\label{tecq_1} Let $\varepsilon, M>0$ be as in \eqref{sup} and $q\beta_{\ell^\star} < 1$.
\begin{itemize} \item For $\ell_1, \dots, \ell_q \in \mathcal I^\star$,  $\beta_{\ell_j}<1$ for every $j$ and $T>\max(1,M)$ we have 
\begin{eqnarray*}
\frac{k_{\ell_1\dots\ell_q}(T)}{T^{2-q\beta_{\ell^\star}}}\le 2C_{\ell_1}(0) \cdots C_{\ell_q}(0) \left ( M + \frac{(\varepsilon +1)^q}{1-q\beta_{\ell^\star}} \left ( 1 + \frac1M \right )^{1-q\beta_{\ell^\star}}\right ).
\end{eqnarray*}
\item For $(\ell_1,\dots,\ell_q)\notin (\mathcal I^\star)^q$, $\ell_j\ge 1$, $\beta_{\ell_j}<1$ for every $j$ and $T>\max(1,M, T_m)$ 
\begin{eqnarray*}
\frac{k_{\ell_1\dots \ell_q}(T)}{T^{2-q\beta_{\ell^\star}}} &\le& 2C_{\ell_1}(0)\dots C_{\ell_q}(0)\Big (M\\
&&+  \frac{(\varepsilon+1)^q}{-(\beta_{\ell^{\star\star}} + (q-1) \beta_{\ell^\star})+1}\left (1 + \frac1M      \right )^{-(\beta_{\ell^{\star\star}} + (q-1) \beta_{\ell^\star})+1}\mathbf{1}_{\beta_{\ell^{\star\star}} + (q-1) \beta_{\ell^\star}<1}\\
&&+ 2(\varepsilon +1)^q m(\beta_{\ell^\star}) \mathbf{1}_{\beta_{\ell^{\star\star}} + (q-1) \beta_{\ell^\star}=1}\\
&&+ \frac{(\varepsilon +1)^q}{\beta_{\ell^{\star\star}} + (q-1) \beta_{\ell^\star}-1}\left (\frac{1}{M+1}   \right )^{\beta_{\ell^{\star\star}} + (q-1) \beta_{\ell^\star}-1}\mathbf{1}_{\beta_{\ell^{\star\star}} + (q-1) \beta_{\ell^\star}>1}\Big ).
\end{eqnarray*} 
\end{itemize}
\end{lemma}

\begin{lemma}\label{tecq_2}
Let $\varepsilon, M>0$ be as in \eqref{sup} and  $q\beta_{\ell^\star} =1$. 
\begin{itemize}
\item For $\ell_1, \dots, \ell_q \in \mathcal I^\star$, $\beta_{\ell_j}<1$ for every $j$ and  $T>\max(1,M,e)$ 
\begin{eqnarray*}
\frac{k_{\ell_1\dots \ell_q}(T)}{T\log T} \le 2C_{\ell_1}(0)\cdots C_{\ell_q}(0) ( M +\log(e+1)). 
\end{eqnarray*}
\item For $(\ell_1, \dots, \ell_q)\notin( \mathcal I^\star)^q$, $\ell_j\ge 1$, $\beta_{\ell_j}<1$ for every $j$ and  $T>\max(1,M,e)$
 \begin{eqnarray*}
\frac{k_{\ell_1\dots \ell_q}(T)}{T\log T} \le 2C_{\ell_1}(0)\cdots C_{\ell_q}(0)\left ( M +(\varepsilon +1)^q \int_{\mathbb R} (1 +|\tau|)^{-(\beta_{\ell^{\star\star}} + (q-1)\beta_{\ell^\star})}\,d\tau  \right ).
\end{eqnarray*} 
\end{itemize}
\end{lemma}

\begin{lemma}\label{tecq_3}
Let $\varepsilon, M>0$ be as in \eqref{sup} and  $q\beta_{\ell^\star} >1$. Then for $T>\max(1,M)$ 
\begin{eqnarray*}
\frac{ k_{\ell _{1}\ell _{2}...\ell_{q}}(T) }{T} \le 2 C_{\ell_1}(0)\cdots C_{\ell_q}(0)\left ( M + \frac{(1+M)}{q\beta_{\ell^\star}-1} \left ( \frac{1+\varepsilon}{(1+M)^{\beta_{\ell^\star}}}\right )^q\right )
\end{eqnarray*}
for any $\ell_1, \dots, \ell_q$ such that $\beta_{\ell_j}<1$, $\ell_j\ge 1$ for every $j$.
\end{lemma}

\begin{lemma}\label{00}
Let $\varepsilon, M>0$ be as in \eqref{sup}, and set $U:=U(\ell_1,\dots,\ell_q)=\lbrace j \in \lbrace 1,\dots, q\rbrace : \ell_j = 0\rbrace$. If $\beta_{\ell^\star}\le \beta_0$
\begin{eqnarray*}
k_{\ell_1 \dots \ell_{q-\#U} 0\dots 0}(T) &\le& 2T C_{\ell_1}(0)\cdots C_{\ell_{q-\#U}}(0)C_0(0)^{\#U}\cr
&&\times \left ( M + (\varepsilon+1)^q\int_{[M,T]} (1 +|\tau|)^{-((q-1)\beta_{\ell^\star}+\beta_0)}\,d\tau\right ),
\end{eqnarray*}
otherwise if $\beta_{\ell^\star} > \beta_0$
\begin{flalign*}
&k_{\ell_1 \dots \ell_{q-\#U} 0\dots 0}(T) \le 2T C_{\ell_1}(0)\cdots C_{\ell_{q-\#U}}(0)C_0(0)^{\#U}\left ( M + (\varepsilon+1)^q\int_{[M,T]} (1 +|\tau|)^{-q\beta_0}\,d\tau\right ).
\end{flalign*}
\end{lemma} 
We are now in the position to prove Proposition \ref{prop-varqth}. 

\begin{proof}[Proof of Proposition \ref{prop-varqth}]
Note first that
\begin{equation}\label{serieqfin} 
\sum_{\ell
_{1},\dots ,\ell _{q}=0}^{\infty }\left ( \prod\limits_{i=1}^{q}\sqrt{\frac{2\ell _{i}+1}{4\pi }}C_{\ell_i}(0)\right )
\mathcal{G}_{\ell _{1}\dots \ell _{q}}^{0\,...0} < +\infty\,,
\end{equation}
since the following estimate (see \cite[\S 4.2.1]{MV16}) 
\begin{equation}\label{gauntQ}
 \mathcal{G}_{\ell _{1}...\ell _{q}}^{0\,...0} \le \sqrt{\frac{(2\ell_1+1)(2\ell_2+1)\cdots (2\ell_{q-1}+1)}{(4\pi)^{q-2}(2\ell_q + 1)}}
 \end{equation}
and \eqref{eqConv} hold. 
Now, assume $q\beta_{\ell^\star} <1$ and $\beta_{\ell^\star}\le \beta_0$ and recall equation \eqref{eq:varqth}. Let $J:=\lbrace (\ell_1,\dots, \ell_q) \text{: there is at least one index } j\in\lbrace 1,\dots q\rbrace \text{ such that } \ell_j=0\rbrace$. Then, thanks to Lemma \ref{00}, we can apply Dominated Convergence Theorem and then Lemma \ref{lemma-varqth} to get
\begin{eqnarray}
&&\lim_{T\to\infty} \sum_{(\ell_1,\dots\ell_q)\in J} \frac{k_{\ell _{1}\dots \ell_{q}}(T)}{T^{2-q\beta_{\ell^\star}}}
\left ( \prod\limits_{i=1}^{q}\sqrt{\frac{2\ell _{i}+1}{4\pi }}\right )
\mathcal{G}_{\ell _{1}\ell _{2}...\ell _{q}}^{0\,...0} \cr
&&= \begin{cases}
0 \  &\text{if } \beta_{\ell^\star} < \beta_0\,,\cr
\sum_{(\ell_1,\dots,\ell_q)\in ( \mathcal I^\star)^q\cap J }  \frac{C_{\ell_1}(0)\cdots C_{\ell_q}(0)}{(1-q\beta_{\ell^\star})(2-q\beta_{\ell^\star})}
\left ( \prod\limits_{i=1}^{q}\sqrt{\frac{2\ell _{i}+1}{4\pi }}\right )
\mathcal{G}_{\ell _{1}\dots \ell _{q}}^{0\,...0}
\  &\text{if } \beta_{\ell^\star} = \beta_0\,.
\end{cases} 
\end{eqnarray}
Now, thanks to Lemma \ref{lemmaqq}, Lemma \ref{tecq_1}, Lemma \ref{00} and \eqref{serieqfin} 
 together with Lemma \ref{lemma-varqth} we have
\begin{eqnarray}\label{scambio1}
&&\lim_{T\to\infty} \sum_{(\ell_1,\dots,\ell_q)\in (\mathcal{I}^\star)^q}\frac{k_{\ell _{1}\ell _{2}...\ell_{q}}(T)}{T^{2-q\beta_{\ell^\star}}}
\left ( \prod\limits_{i=1}^{q}\sqrt{\frac{2\ell _{i}+1}{4\pi }}\right )
\mathcal{G}_{\ell _{1}\ell _{2}...\ell _{q}}^{0\,...0}\notag \\
&&=\sum_{(\ell_1,\dots,\ell_q)\in (\mathcal{I}^\star)^q}\lim_{T\rightarrow\infty} \frac{k_{\ell _{1}\ell _{2}...\ell_{q}}(T)}{T^{2-q\beta_{\ell^\star}}}
\left ( \prod\limits_{i=1}^{q}\sqrt{\frac{2\ell _{i}+1}{4\pi }}\right )
\mathcal{G}_{\ell _{1}\ell _{2}...\ell _{q}}^{0\,...0}\notag\\
&&=\sum_{(\ell_1,\dots,\ell_q)\in (\mathcal{I}^\star)^q}  \frac{C_{\ell_1}(0)\cdots C_{\ell_q}(0)}{(1-q\beta_{\ell^\star})(2-q\beta_{\ell^\star})}
\left ( \prod\limits_{i=1}^{q}\sqrt{\frac{2\ell _{i}+1}{4\pi }}\right )
\mathcal{G}_{\ell _{1}\ell _{2}...\ell _{q}}^{0\,...0}.
\end{eqnarray} 
Analogously
\begin{eqnarray}\label{scambio2}
&&\lim_{T\to\infty} \sum_{(\ell_1,\dots,\ell_q)\notin (\mathcal{I}^\star)^q}\frac{k_{\ell _{1}\ell _{2}...\ell_{q}}(T)}{T^{2-(\beta_{\ell_1}+\cdots+\beta_{\ell_q})}}
\left ( \prod\limits_{i=1}^{q}\sqrt{\frac{2\ell _{i}+1}{4\pi }}\right )
\mathcal{G}_{\ell _{1}\ell _{2}...\ell _{q}}^{0\,...0} =0\,.
\end{eqnarray} 
Le us check \eqref{scambio2}. We have 
\begin{eqnarray*}
&&\lim_{T\rightarrow\infty}\sum_{\substack{(\ell_{1},\ell _{2},\dots ,\ell _{q})\notin \mathcal{I}^\star:\\ \beta_{\ell_1}+\cdots+\beta_{\ell_q}<1, \ell_j\ge 1} }\frac{k_{\ell _{1}\ell _{2}...\ell_{q}}(T)}{T^{2-q\beta_{\ell^\star}}} \left (\prod\limits_{i=1}^{q}\sqrt{\frac{2\ell _{i}+1}{4\pi }}\right) \mathcal{G}_{\ell _{1}\ell _{2}...\ell _{q}}^{0\,...0}\\
&&=\sum_{\substack{(\ell_{1},\ell _{2},\dots ,\ell _{q})\notin \mathcal{I}^\star:\\ \beta_{\ell_1}+\cdots+\beta_{\ell_q}<1, \ell_j\ge 1}}\lim_{T\rightarrow\infty}\underbrace{\frac{k_{\ell _{1}\ell _{2}...\ell_{q}}(T)}{T^{2-(\beta_{\ell_1}+\cdots+\beta_{\ell_q})}}}_{\rightarrow \frac{C_{\ell_1}(0)\cdots C_{\ell_q}(0)}{(1-(\beta_{\ell_1}+\cdots+\beta_{\ell_q}))(2-(\beta_{\ell_1}+\cdots+\beta_{\ell_q}))}} \cr
&&\times \underbrace{\frac{T^{2-(\beta_{\ell_1}+\cdots+\beta_{\ell_q})}}{T^{2-q\beta_{\ell^\star}}}}_{\rightarrow0} \left (\prod\limits_{i=1}^{q}\sqrt{\frac{2\ell _{i}+1}{4\pi }}\right ) \mathcal{G}_{\ell _{1}\ell _{2}...\ell _{q}}^{0\,...0}=0\,.
\end{eqnarray*}
Analogously 
\begin{eqnarray*}
&&\lim_{T\rightarrow\infty}\sum_{\substack{(\ell_{1},\ell _{2},\dots ,\ell _{q})\notin \mathcal{I}^\star:\\ \beta_{\ell_1}+\cdots+\beta_{\ell_q}=1, \ell_j\ge 1}}\frac{k_{\ell _{1}\ell _{2}...\ell_{q}}(T)}{T^{2-q\beta_{\ell^\star}}} \left (\prod\limits_{i=1}^{q}\sqrt{\frac{2\ell _{i}+1}{4\pi }}\right) \mathcal{G}_{\ell _{1}\ell _{2}...\ell _{q}}^{0\,...0}=0
\end{eqnarray*}
and finally 
\begin{eqnarray*}
&&\lim_{T\rightarrow\infty}\sum_{\substack{(\ell_{1},\ell _{2},\dots ,\ell _{q})\notin \mathcal{I}^\star:\\ \beta_{\ell_1}+\cdots+\beta_{\ell_q}>1, \ell_j\ge 1}}\frac{k_{\ell _{1}\ell _{2}...\ell_{q}}(T)}{T^{2-q\beta_{\ell^\star}}} \left (\prod\limits_{i=1}^{q}\sqrt{\frac{2\ell _{i}+1}{4\pi }}\right) \mathcal{G}_{\ell _{1}\ell _{2}...\ell _{q}}^{0\,...0}=0\,,
\end{eqnarray*}
so that \eqref{scambio2} is proved.  

On the other hand, if we assume $q\beta_0>1$ and $q\beta_{\ell^\star} >1$, then obviously $q\beta_{\ell}>1$ for all $\ell\in \widetilde{\mathbb{N}}$, and $\beta_{\ell_1}+\cdots+\beta_{\ell_q}>1$ for all $\ell_1,\dots,\ell_q \in \widetilde{\mathbb{N}}$. 
Then, thanks to Lemma \ref{lemma-varqth}, Lemma \ref{lemmaqq}, Lemma \ref{tecq_3} and Lemma \ref{00},
\begin{eqnarray*}
\lim_{T\to\infty} \frac{\Var(\mathcal{M}_T(u)[q]))}{T} &=& 4\pi \sum_{\ell
_{1},\ell _{2},\dots ,\ell _{q}=0}^{\infty } \lim_{T\to\infty} \frac{k_{\ell _{1}\ell _{2}...\ell
_{q}}(T)}{T} \left ( \prod\limits_{i=1}^{q}\sqrt{\frac{2\ell _{i}+1}{4\pi }}\right )
\mathcal{G}_{\ell _{1}\ell _{2}...\ell _{q}}^{0\,...0}\\
&=& 4\pi \sum_{\ell
_{1},\ell _{2},\dots ,\ell _{q}=0}^{\infty }\mathcal{G}_{\ell _{1}...\ell _{q}}^{0\,...0} \int_{-\infty}^{+\infty}\left ( \prod\limits_{i=1}^{q}\sqrt{\frac{2\ell _{i}+1}{4\pi }}C_{\ell _{i}}(\tau)\right )\,d\tau \,,
\end{eqnarray*}
 which concludes the proof. 
 In particular, we have proved that the series on the right hand side of the previous formula converges. The remaining cases can be treated analogously. 
 \end{proof}

\subsubsection{Proof of Lemma \ref{lemma-varqth}} \label{secVarq}

\begin{proof} This proof is similar to the one of Lemma \ref{lemma-var2nd}.
Consider $\varepsilon, M>0$
as in \eqref{sup}. 
Then, using Remark \ref{remComput}, we have 
\begin{flalign}
&k_{\ell _{1}\ell _{2}...\ell _{q}}(T)=\int_{[0,T]^{2}}C_{\ell _{1}}(t-s)\,C_{\ell _{2}}(t-s)\cdots C_{\ell_{q}}(t-s)dtds\notag\\
&=2T\int_{0}^{M}\Big(1-\frac{\tau}{T}\Big)C_{\ell _{1}}(\tau) C_{\ell _{2}}(\tau)\cdots C_{\ell_{q}}(\tau)d\tau+2T\int_{M}^{T}\Big(1-\frac{\tau}{T}\Big)C_{\ell _{1}}(\tau) C_{\ell _{2}}(\tau)\cdots C_{\ell_{q}}(\tau)d\tau\,.\qquad\qquad \label{eq3-lemma-varqth}
\end{flalign}
Now assume that $\beta_{\ell_1}+\cdots+\beta_{\ell_q}<1$. For the first summand on the right hand side of \eqref{eq3-lemma-varqth}
we have 
\begin{flalign*}
&\lim_{T\rightarrow\infty}\frac{2}{T^{1-(\beta_{\ell_1}+\cdots+\beta_{\ell_q})}}\int_{0}^{M}\left(1-\frac{\tau}{T}\right)|C_{\ell _{1}}(\tau)\,C_{\ell _{2}}(\tau)\cdots C_{\ell_{q}}(\tau)|d\tau \le \lim_{T\to\infty} 2\frac{C_{\ell _{1}}(0)\cdots C_{\ell_{q}}(0)}{T^{1-(\beta_{\ell_1}+\cdots+\beta_{\ell_q})}} M = 0\,.
\end{flalign*} 
For the second summand on the right hand side of \eqref{eq3-lemma-varqth} we write 
\begin{eqnarray}
&&\int_{M}^{T}\left(1-\frac{\tau}{T}\right)C_{\ell _{1}}(\tau)\,C_{\ell _{2}}(\tau)\cdots C_{\ell_{q}}(\tau)d\tau\notag\\
&&=C_{\ell _{1}}(0)\cdots C_{\ell_{q}}(0) \int_{M}^{T}\left(1-\frac{\tau}{T}\right)(1+\tau)^{-(\beta_{\ell_1}+\cdots+\beta_{\ell_q})}\,\times\notag\\
&&\qquad \times\sum_{k=1}^q\sum_{\substack{k_1+\cdots+k_q=k\\k_1,\dots,k_q\in\{0,1\}}}\left(\frac{G_{\ell _{1}}(\tau)}{C_{\ell _{1}}(0)}-1\right)^{k_1}\cdots \left(\frac{G_{\ell _{q}}(\tau)}{C_{\ell _{q}}(0)}-1\right)^{k_q}d\tau\notag\\
&&\qquad+C_{\ell _{1}}(0)\cdots C_{\ell_{q}}(0) \int_{M}^{T}\left(1-\frac{\tau}{T}\right)(1+\tau)^{-(\beta_{\ell_1}+\cdots+\beta_{\ell_q})}d\tau\,. \label{ciaoq}
\end{eqnarray}
For the first term on the right hand side of the previous equality it holds that
\begin{eqnarray}\label{qaus1}
&&\lim_{T\rightarrow\infty}\frac{C_{\ell _{1}}(0)\cdots C_{\ell_{q}}(0)}{T^{1-(\beta_{\ell_1}+\cdots+\beta_{\ell_q})}} \int_{M}^{T}\left(1-\frac{\tau}{T}\right)(1+\tau)^{-(\beta_{\ell_1}+\cdots+\beta_{\ell_q})}\,\times \notag \\
&&\qquad\times\sum_{k=1}^q\sum_{\substack{k_1+\cdots+k_q=k\\k_1,\dots,k_q\in\{0,1\}}}\left(\frac{G_{\ell _{1}}(\tau)}{C_{\ell _{1}}(0)}-1\right)^{k_1}\cdots \left(\frac{G_{\ell _{q}}(\tau)}{C_{\ell _{q}}(0)}-1\right)^{k_q}d\tau=0\,.
\end{eqnarray}
Let us prove \eqref{qaus1}. Actually, for $\tau >M$ we have 
$$
\left|\sum_{k=1}^q\sum_{\substack{k_1+\cdots+k_q=k\\k_1,\dots,k_q\in\{0,1\}}}\left(\frac{G_{\ell _{1}}(\tau)}{C_{\ell _{1}}(0)}-1\right)^{k_1}\cdots \left(\frac{G_{\ell _{q}}(\tau)}{C_{\ell _{q}}(0)}-1\right)^{k_q}\right| \le \sum_{k=1}^q  {q\choose k}  \varepsilon^k\,,
$$
and \eqref{qaus1} follows, $\varepsilon$ being arbitrary.
On the other hand, for the second summand on the right hand side of \eqref{ciaoq},
\begin{eqnarray*}
&&\lim_{T\rightarrow\infty}\frac{C_{\ell _{1}}(0)\cdots C_{\ell_{q}}(0) }{T^{1-(\beta_{\ell_1}+\cdots+\beta_{\ell_q})}}  \int_{M}^{T}\left(1-\frac{\tau}{T}\right)(1+\tau)^{-(\beta_{\ell_1}+\cdots+\beta_{\ell_q})}d\tau\\
&&=\frac{C_{\ell _{1}}(0)\cdots C_{\ell_{q}}(0)}{(1-(\beta_{\ell_1}+\cdots+\beta_{\ell_q}))(2-(\beta_{\ell_1}+\cdots+\beta_{\ell_q}))}\,.
\end{eqnarray*} 
Analogously, if $\beta_{\ell_1}+\cdots+\beta_{\ell_q}=1$
$$
\lim_{T\to\infty} \frac{k_{\ell_1\dots\ell_q}(T)}{T\log T} = 2 \,C_{\ell _{1}}(0)\cdots C_{\ell_{q}}(0)\,.
$$
Otherwise, if $\beta_{\ell_1}+\cdots+\beta_{\ell_q}>1$, it immediately follows from equation \eqref{eq3-lemma-varqth} that, as $T\to +\infty$, 
$$
k_{\ell_1\dots\ell_q}(T)= 2T\int_0^{+\infty} C_{\ell _{1}}(\tau)\,C_{\ell _{2}}(\tau)\cdots C_{\ell_{q}}(\tau)\,d\tau + O(1)\,.
$$
Note  that the limiting  constant 
$$
\int_{\mathbb R} C_{\ell _{1}}(\tau)\,C_{\ell _{2}}(\tau)\cdots C_{\ell_{q}}(\tau)\,d\tau
$$
 in \eqref{eq2-lemma-varqth} is finite (see the proof of Proposition \ref{prop-varqth}).
 \end{proof}

\section{Proofs of the main results}

\subsection{Proof of Theorem \ref{uno}}

\begin{proof}
Recall \eqref{var_M}. Assume first that $u\ne 0$ and $\beta _{0}<\min (2\beta _{\ell ^{\star }},1)$. For the first chaotic projection, since $\beta_0<1$, from Lemma \ref{lem1} we have 
\begin{equation}\label{lim1}
\lim_{T\to\infty} \frac{\Var(\mathcal M_T(u)[1])}{T^{2-\beta_0}} = \frac{2\phi(u)^2 C_0(0)}{(1-\beta_0 )(2-\beta_0 )}\,.
\end{equation}
Let $Q\in \lbrace 2,3,\dots \rbrace$ be such that 
$$
Q\beta_{\ell^\star} >1\,.
$$
For $q\in \lbrace 2,3,\dots, Q-1\rbrace$ we have, from Proposition \ref{prop-var2nd} and Proposition \ref{prop-varqth},
since $\beta_0 < 2\beta_{\ell^\star}$, 
\begin{equation}\label{lim2}
\lim_{T\rightarrow \infty}\sum_{q=2}^{Q-1} \frac{\Var\left(\mathcal{M}_T(u)[q]\right)}{T^{2-\beta_{0}}} =\sum_{q=2}^{Q-1}  \lim_{T\rightarrow \infty}\frac{\Var\left(\mathcal{M}_T(u)[q]\right)}{T^{2-\beta_{0}}}=0\,.
\end{equation}
Let us now prove that 
\begin{eqnarray}\label{eq1-proofthm1}
\lim_{T\rightarrow \infty}\sum_{q = Q}^{+\infty} \frac{\Var\left(\mathcal{M}_T(u)[q]\right)}{T^{2-\beta_{0}}} = 0\,.
\end{eqnarray}
Recall \eqref{eq:varqth}; thanks to \eqref{gauntQ} we can write for any $\ell_1,\dots, \ell_q\ge 0$ 
\begin{flalign*}
&\frac{4\pi\,H_{q-1}(u)^2\phi(u)^2}{q!}\left ( \prod\limits_{i=1}^{q}\sqrt{\frac{2\ell _{i}+1}{4\pi }}\right) \mathcal{G}_{\ell _{1}\dots \ell _{q}}^{0\,...0}\\
&\le \frac{(4\pi)^2 H_{q-1}(u)^2\phi(u)^2}{q!} \left ( \prod_{i=1}^{q} \frac{2\ell_i+1}{4\pi} \right )=: b_q(\ell_1,\dots,\ell_q;u). 
\end{flalign*}
For $q\ge Q$ we have of course $q\beta_{\ell^\star}>1$. Let $\varepsilon, M>0$ be as in \eqref{sup}. From Lemma \ref{tecq_3} we have for $T>\max(1,M)$ 
\begin{flalign}
&\sum_{\substack{\ell_{1},\ell _{2},\dots ,\ell _{q}\ge 1\\ \beta_{\ell_1},\dots,\beta_{\ell_q}<1} }b_q(\ell_1,\dots,\ell_q;u)\frac{k_{\ell_1,\dots,\ell_q}(T)}{T^{2-\beta_0}} \le\sum_{\substack{\ell_{1},\ell _{2},\dots ,\ell _{q}\ge 1\\ \beta_{\ell_1},\dots,\beta_{\ell_q}<1} }b_q(\ell_1,\dots,\ell_q;u) \frac{k_{\ell_1,\dots,\ell_q}(T)}{T} \cr
&\le 2\sum_{\substack{\ell_{1},\ell _{2},\dots ,\ell _{q}\ge 1\\ \beta_{\ell_1},\dots,\beta_{\ell_q}<1} }b_q(\ell_1,\dots,\ell_q;u) C_{\ell_1}(0)\cdots C_{\ell_q}(0)\left ( M + \frac{(1+M)}{q\beta_{\ell^\star}-1} \left ( \frac{1+\varepsilon}{(1+M)^{\beta_{\ell^\star}}}\right )^q\right )\cr
&\le 2\sum_{\ell_{1},\ell _{2},\dots ,\ell _{q}\ge 0 }b_q(\ell_1,\dots,\ell_q;u) C_{\ell_1}(0)\cdots C_{\ell_q}(0)\left ( M + \frac{(1+M)}{q\beta_{\ell^\star}-1} \left ( \frac{1+\varepsilon}{(1+M)^{\beta_{\ell^\star}}}\right )^q\right )\cr
&= 2 \frac{(4\pi)^2 H_{q-1}(u)^2\phi(u)^2}{q!}\left ( M + \frac{(1+M)}{q\beta_{\ell^\star}-1} \left ( \frac{1+\varepsilon}{(1+M)^{\beta_{\ell^\star}}}\right )^q\right ),\label{termprinc}
\end{flalign} 
recalling \eqref{somma1}. The following estimate holds (see e.g.~\cite[Proposition 3]{IPV95}): for every $q\ge 0$ and $x\in \mathbb R$
\begin{equation*}
| \text{e}^{-x^2/4} H_q(x)| \le c \sqrt{q!}\,q^{-1/12},
\end{equation*}
hence the series whose term is the right hand side of \eqref{termprinc} is finite, i.e.,
\begin{equation}\label{CD}
\sum_{q=Q}^{+\infty}  \frac{H_{q-1}(u)^2\phi(u)^2}{q!}\left ( M + \frac{(1+M)}{q\beta_{\ell^\star}-1} \left ( \frac{1+\varepsilon}{(1+M)^{\beta_{\ell^\star}}}\right )^q\right ) < +\infty\,,
\end{equation} 
as soon as $M$ is sufficiently large. Repeating the same argument as for \eqref{termprinc}, using Lemma \ref{lemmaqq} and Lemma \ref{00}, and thanks to \eqref{CD}, we can apply Dominated Convergence Theorem and then Proposition \ref{prop-var2nd} and Proposition \ref{prop-varqth} to get 
\begin{equation*}
\lim_{T\rightarrow \infty}\sum_{q = Q}^{+\infty} \frac{\Var\left(\mathcal{M}_T(u)[q]\right)}{T^{2-\beta_{0}}} = \sum_{q = Q}^{+\infty} \lim_{T\rightarrow \infty}\frac{\Var\left(\mathcal{M}_T(u)[q]\right)}{T^{2-\beta_{0}}} =0\,,
\end{equation*} 
which is \eqref{eq1-proofthm1}. Putting together \eqref{lim1}, \eqref{lim2} and  \eqref{eq1-proofthm1})we finally find that
\begin{eqnarray*}
\lim_{T\to\infty} \frac{\Var(\mathcal M_T(u))}{T^{2-\beta_0}} = \frac{2\phi(u)^2 C_0(0)}{(1-\beta_0 )(2-\beta_0 )}=:K_0(u).
\end{eqnarray*}
Note that, if $u=0$, then $\mathcal{M}_T(u)[2]\equiv 0$ and the sufficient condition in order to have \eqref{eq1-proofthm1} is $\beta _{0}<\min (3\beta _{\ell ^{\star }},1)$. 

This implies that, if either $u\ne0$ and $\beta _{0}<\min (2\beta _{\ell ^{\star }},1)$ or $u=0$ and $\beta _{0}<\min (3\beta _{\ell ^{\star }},1)$, then
$$
\widetilde{\mathcal{M}}_T(u)=\frac{\mathcal{M}_T(u)[1]}{\sqrt{K_0(u)}\,T^{1-\beta_{0}/2}}+o_{\mathbb{P}}(1)\,.
$$
Consequently, since $\mathcal{M}_T(u)[1]$ is Gaussian for any $T>0$, it is clear that the asymptotic distribution of $\widetilde{\mathcal{M}}_T(u)$ is standard Gaussian.
\end{proof} 

\subsection{Proof of Theorem \ref{due}}

We will need the following well known  result. 
\begin{theorem}[\cite{DM79, Ta:79}] 
Let $\xi(t)$, $t\in\mathbb{R}$, be a real measurable mean-square continuous
stationary Gaussian process with mean $\mathbb{E}\left[\xi(t)\right]$ and
covariance function $\rho(t-s)=\rho(|t-s|)=\Cov(\xi(t),\xi(s))$.
Moreover, assume that 
\begin{equation}\label{eq:DM79}
\rho(t-s)=\frac{L(|t-s|)}{|t-s|^{\beta}}\,, \qquad \text{with} \quad
0<\beta<1,
\end{equation}
where $L$ is a slowly varying function. 
Let $F: \mathbb{R} \rightarrow \mathbb{R}$ be a Borel function such that $%
\mathbb{E}\left[F(N)^2\right]<+\infty$, where $N$ is a standard Gaussian
random variable. Then it is a well known fact that can be expanded as
follows 
\begin{equation*}
F(\xi)=\sum_{k=0}^\infty \frac{b_k}{k!}H_k(\xi)\,, \quad \text{where} \quad
b_k=\int_{\mathbb{R}}F(\xi)H_k(\xi)\phi(\xi)d\xi\,.
\end{equation*}
Assume there exists an integer $r$, the so-called Hermitian rank, such that $%
b_0=b_1=\cdots=b_{r-1}=0$ and $b_r\ne0$. Then , if $\beta\in(0,1/r)$, we
have that the finite-dimensional distributions of the random process 
\begin{equation*}
X_T(s)=\frac{1}{T^{1-\beta r /2}L(T)^{r/2}}\int_{0}^{Ts}\,\left[F(\xi(t))-b_0%
\right] \, dt \,, \qquad 0\le s \le1\,,
\end{equation*}
converge weakly, as $T\rightarrow \infty$, to the ones of the Rosenblatt
process of order $r$, that is 
\begin{equation*}
X_{\beta}(s):=\frac{b_r}{r!}\int_{(\mathbb{R}^r)'}\,\frac{e^{i(\lambda_1+\cdots+%
\lambda_r)s}-1}{i(\lambda_1+\cdots+\lambda_r)}\frac{W(d\lambda_1)\cdots
W(d\lambda_r)}{|\lambda_1\cdots\lambda_r|^{(1-\beta)/2}} \, dt \,, \qquad
0\le s \le1\,,
\end{equation*}
where $W$ is a complex Gaussian white noise.
\end{theorem}

\begin{proof}[Proof of Theorem \ref{due}] 
Recall that $u\ne0$ and $2\beta_{\ell^\star}<\min(\beta_0,1)$.  From Lemma \ref{lem1} we have 
$$
\lim_{T\to\infty} \frac{\Var(\mathcal M_T(u)[1])}{T^{2-2\beta_{\ell^\star}}} = 0\,.
$$
Moreover, thanks to Proposition \ref{prop-varqth}, as for the proof of Theorem \ref{uno} (in particular \eqref{eq1-proofthm1}), we have, as $T\to +\infty$, 
$$
\lim_{T\to\infty} \frac{\sum_{q\ge 3} \Var(\mathcal M_T(u)[q])}{T^{2-2\beta_{\ell^\star}}} = 0\,,
$$
so that, recalling also Proposition \ref{prop-var2nd}, 
\begin{equation}\label{2chaosA}
\frac{\mathcal{M}_T(u)}{T^{1-\beta_{\ell^\star}}}=\frac{\mathcal{M}_T(u)[2]}{T^{1-\beta_{\ell^\star}}}+o_{\mathbb P}(1)\,.
\end{equation}
Moreover, since in $L^2(\Omega)$ we have the following equality
\begin{equation}\label{eq-2ndALM}
\mathcal{M}_T(u)[2]=\frac{J_2(u)}{2}\sum_{\ell=0}^\infty\sum_{m=-\ell}^\ell\,\int_0^T
\,H_{2}(a_{\ell m}(t))dt\,,
\end{equation} 
it holds that
\begin{equation}  \label{eq:var2spectral}
\frac{\mathcal{M}_T(u)[2]}{T^{1-\beta_{\ell^\star}}}=\frac{1}{T^{1-\beta_{\ell^\star}}}\sum_{\ell\in \mathcal{I}^\star}\sum_{m=-\ell}^\ell \frac{J_2(u)}{2}\int_0^T H_2(a_{\ell, m}(t))\, dt +o_{\mathbb P}(1)\,.
\end{equation}
Indeed, reasoning exactly as in the Proposition \ref{prop-var2nd}, we have that 
\begin{eqnarray*}
&&\lim_{T\to\infty}\mathbb E\left [\left(\frac{\mathcal{M}_T(u)[2]}{T^{1-\beta_{\ell^\star}}} -\frac{1}{T^{1-\beta_{\ell^\star}}}\sum_{\ell\in \mathcal{I}^\star}\sum_{m=-\ell}^\ell \frac{J_2(u)}{2}\int_0^T H_2(a_{\ell, m}(t))\, dt \right)^2\right ]\\
&&=\lim_{T\to\infty} \frac{J_2(u)^2}{2T^{2-2\beta_{\ell^\star}}}\sum_{\ell\notin \mathcal I^\star} (2\ell+1)
\int_0^T\int_0^TC_\ell(t-s)^2dt\,ds =0\,.
\end{eqnarray*}
From \eqref{2chaosA} and \eqref{eq:var2spectral}, in order to understand the asymptotic distribution of $\mathcal M_T(u)$, it suffices to investigate the leading term on the right hand side of \eqref{eq:var2spectral}. 
Recall Condition \ref{basic2}, for $\ell\in \mathcal I^\star$ we have that 
$$
C_{\ell}(\tau)= \frac{G_\ell(\tau)}{(1+|\tau|)^{\beta_{\ell^\star}}} \,,
$$
where in particular $G_\ell$ is a slowly varying function.
Hence, setting $\xi(t)=a_{\ell, m}(t)$, we automatically have that $\rho=\rho_\ell=C_{\ell}$, $L=L_\ell=G_\ell$ and, as a consequence, that
\begin{equation*}
X_T^{\ell, m}:=\frac{1}{C_\ell(0)\,T^{1-\beta_{\ell^\star}}}\int_0^T \frac{J_2(u)}{2}H_2(a_{\ell, m}(t))\, dt 
\overset{d}{\longrightarrow} \frac{J_2(u)}{2a(\beta_{\ell^\star})}\,X_{m;\beta_{\ell^\star}}\,, \qquad \text{as } T \rightarrow
\infty\,,
\end{equation*}
for all $m=-\ell,\dots,\ell$, where for each $m$
$X_{m;\beta_{\ell^\star}}$
is a standard Rosenblatt random variable \eqref{Xbeta} of parameter $\beta_{\ell^\star}$. 
Moreover, since the $X_T^{\ell, m}$ are all independent for each $T$ we have that 
\begin{eqnarray*}
&&\widetilde{\mathcal{M}}_T(u)
=\sqrt{\frac{T^{1-\beta_{\ell^\star}}}{\Var\left(\mathcal{M}_T(u)[2]\right)}}\sum_{\ell\in \mathcal{I}^\star}C_\ell(0)\sum_{m=-\ell}^\ell 
\frac{\int_0^T\frac{J_2(u)}{2} H_2(a_{\ell, m}(t))\, dt}{C_\ell(0)\,T^{1-\beta_{\ell^\star}}}+o_{\mathbb{P}}(1)\\
&&\mathop{\to}^d \left (\frac{J_2(u)^2}{2}\sum_{\ell\in \mathcal{I}^\star}\frac{(2\ell+1)C_\ell(0)^2}{(1-\beta_{\ell^\star})(1-2\beta_{\ell^\star})}\right )^{-1/2}\sum_{\ell\in \mathcal{I}^\star}C_\ell(0)\sum_{m=-\ell}^{\ell}\frac{J_2(u)}{2a(\beta_{\ell^\star})}X_{m;\beta_\ell^\star}\\
&&= \sum_{\ell\in \mathcal{I}^\star}\frac{C_\ell(0)}{\sqrt{v^\star}}\sum_{m=-\ell}^{\ell}X_{m;\beta_{\ell^\star}},
\end{eqnarray*}
where 
$$
v^\star= a(\beta_{\ell^\star})^2\sum_{\ell\in \mathcal{I}^\star}\frac{2\,(2\ell+1)\,C_\ell(0)^2}{(1-\beta_{\ell^\star})(1-2\beta_{\ell^\star})}\,,
$$
and the proof is concluded.
\end{proof}

\subsection{Proof of Theorem \ref{tre}}

First of all assume that $3\beta _{\ell ^{\star }}<\min (1,\beta _{0})$. Since we are in the case where $u=0$, we have that all even chaotic projections vanish and hence that
$$
\Var\left(\mathcal{M}_T\right)=\Var\left(\mathcal{M}_T[1]\right)+\Var\left(\mathcal{M}_T[3]\right)+\sum_{q\ge2}\Var\left(\mathcal{M}_T[2q+1]\right)\,.
$$
where we used the notation $\mathcal{M}_T(0)=:\mathcal{M}_T$.
As a consequence,  as in the proof of Theorem \ref{uno}, we have
\begin{flalign*}
&\lim_{T\rightarrow\infty}\frac{\Var\left(\mathcal{M}_T\right)}{T^{2-3\beta_{\ell^\star}}}=\lim_{T\rightarrow\infty}\frac{\Var\left(\mathcal{M}_T[1]\right)}{T^{2-3\beta_{\ell^\star}}}+\lim_{T\rightarrow\infty}\frac{\Var\left(\mathcal{M}_T[3]\right)}{T^{2-3\beta_{\ell^\star}}}
+\sum_{q\ge2}\lim_{T\rightarrow\infty}\frac{\Var\left(\mathcal{M}_T[2q+1]\right)}{T^{2-3\beta_{\ell^\star}}}\,.
\end{flalign*}
Now,
\begin{eqnarray*}
&&\lim_{T\rightarrow\infty}\frac{\Var\left(\mathcal{M}_T[1]\right)}{T^{2-3\beta_{\ell^\star}}}=0\,;
\end{eqnarray*}
while from Proposition \ref{prop-varqth} we know that
\begin{flalign*}
&\lim_{T\rightarrow \infty }\frac{\Var\left(\mathcal{M}_T[3]\right)}{T^{2-3\beta _{\ell^\star}}} =\frac{2}{3!(1-3\beta_{\ell^\star})(2-3\beta_{\ell^\star})}\sum_{\ell
_{1},\ell _{2},\ell _{3}\in \mathcal{I}^\star} \left ( \prod\limits_{i=1}^{3}\sqrt{\frac{2\ell _{i}+1}{4\pi }} C_{\ell_i}(0)\right)
\mathcal{G}_{\ell _{1}\ell _{2}\ell _{3}}^{000}=: K_3.
\end{flalign*}
Moreover
\begin{eqnarray*}
&&\lim_{T\rightarrow \infty }\frac{\sum_{q\ge2}\Var\left(\mathcal{M}_T[2q+1]\right)}{T^{2-3\beta_{\ell^\star}}} =0\,,
\end{eqnarray*}
which of course implies that
$$
\widetilde{\mathcal{M}}_T(u)=\frac{\mathcal{M}_T(u)[3]}{\sqrt{K_3}T^{1 - \frac32 \beta_{\ell^\star}}}+o_{\mathbb P}(1)\,,
$$
as claimed.

\subsection{Proof of Theorem \ref{quattro}}

The result below is just \cite[Theorem 6.3.1]{noupebook} restated for our framework as a lemma. Recall the definition of cumulants for a random variable \cite[\S 4.3]{MaPeCUP}. 
\begin{lemma}\label{lemma:NP12}
Assume that the functional $\widetilde{\mathcal{M}}_T(u)$ in \eqref{eq:stdM} satisfies the following conditions:
\begin{itemize}
\item[(a)] For each $q\ge1$, $\Var(\widetilde{\mathcal{M}}_T(u)[q])\rightarrow\sigma_q^2$, as $T\rightarrow\infty$ and for some $\sigma_q^2\ge0$;
\item[(b)] $\displaystyle{\sigma^2 :=\sum\limits_{q=1}^\infty \sigma_q^2<+\infty}$;
\item[(c)] For each $q\ge2$, $\Cum_4(\widetilde{\mathcal{M}}_T(u)[q])\rightarrow0$, as $T\rightarrow\infty$;
\item[(d)] $\displaystyle{\lim_{Q\rightarrow\infty}\sup_{T>0}\sum_{q=Q+1}^\infty\Var(\widetilde{\mathcal{M}}_T(u)[q])=0}$.
\end{itemize}
Then $\widetilde{\mathcal{M}}_T(u)\stackrel{d}{\to} Z$, as $T\rightarrow\infty$, where $Z\sim \mathcal N(0,\sigma^2)$.
\end{lemma}
We will use Lemma \ref{lemma:NP12} to prove Theorem \ref{quattro}. Let us first focus on Condition \emph{(c)}.
\begin{proposition}\label{prop_aus}
Assume $\beta_0=1$. If either $u\ne0$ and $2\beta _{\ell^\star}>1$ or $u=0$ and $3\beta _{\ell^\star}>1$ we have 
\begin{equation*}
\widetilde{\mathcal{M}}_T(u)[q]\stackrel{d}{\to} Z\,, \quad \text{ as } \,\,T\rightarrow \infty,
\end{equation*}
where $Z\sim \mathcal N(0,\sigma_q^2)$ is a standard Gaussian random variable whose variance is given by 
$$
\sigma_q^2 := \frac{s^2_q}{\sum_{k=1}^{+\infty} s^2_k}\in [0,+\infty),
$$
where the sequence $\lbrace s^2_k, k\ge 1\rbrace$ is defined in Theorem \ref{quattro}.
\end{proposition}

\begin{remark}
Note that some of the chaoses might converge to a degenerate Gaussian (that is, with zero
expected value and variance).
\end{remark}

\begin{proof}[Proof of Proposition \ref{prop_aus}] It suffices to check \cite[Theorem 5.2.7]{noupebook} that the fourth cumulant goes to zero as $T\to +\infty$, i.e.
$$
\lim_{T\to\infty} \Cum_4 \left (\widetilde{\mathcal{M}}_T(u)[q] \right ) = 0\,.
$$
Recall \eqref{def_q}. 
For any $1\leq \alpha \leq q-1$, we have (see e.g.~\cite[\S 4.3]{MaPeCUP}, in particular \S 4.3.1)
\begin{flalign*}
&\Cum_{4}\left ( \int_{0}^{T}\int_{\mathbb{S}^{2}}H_{q}(Z(x,t))dxdt\right ) = \int_{[0,T]^{4}}\int_{(\mathbb{S}^{2})^4}\,dx_{1}dx_{2}dx_{3}dx_{4}dt_{1}dt_{2}dt_{3}dt_{4}\\
&\qquad\times  \Cum\Big (H_{q}(Z(x_{1},t_{1})),H_{q}(Z(x_{2},t_{2}))H_{q}(Z(x_{3},t_{3}))H_{q}(Z(x_{4},t_{4})) \Big )
\\
&\leq c\,\int_{[0,T]^{4}}\int_{(\mathbb{S}^{2})^4}\left |
\mathbb{E}\left[ Z(x_{1},t_{1})Z(x_{2},t_{2})\right] \right |^{q-\alpha
}\left | \mathbb{E}\left[ Z(x_{2},t_{2})Z(x_{3},t_{3})\right] \right |^{\alpha } \\
&\qquad\times \left | \mathbb{E}\left[ Z(x_{3},t_{3})Z(x_{4},t_{4})\right]
\right |^{q-\alpha }\left | \mathbb{E}\left[ Z(x_{4},t_{4})Z(x_{1},t_{1})%
\right] \right|^{\alpha }dx_{1}dx_{2}dx_{3}dx_{4}dt_{1}dt_{2}dt_{3}dt_{4}.
\end{flalign*}%
For $x,y$ positive numbers, it holds that
\begin{equation*}
x^{\alpha }y^{\beta }\leq x^{\alpha +\beta }+y^{\alpha +\beta }\,,
\end{equation*}%
as a consequence, 
\begin{flalign*}
&\Cum_{4}\left ( \int_{0}^{T}\int_{\mathbb{S}^{2}}H_{q}(Z(x,t))dxdt\right ) \\
&\leq c\,\int_{[0,T]^{4}}\int_{(\mathbb{S}^{2})^4}\left | \mathbb{E%
}\left[ Z(x_{1},t_{1})Z(x_{2},t_{2})\right] \right|^{q-\alpha }\left |
\mathbb{E}\left[ Z(x_{2},t_{2})Z(x_{3},t_{3})\right] \right |^{\alpha}\\
&\qquad \qquad\qquad \qquad\qquad \qquad \times \left | \mathbb{E}\left[ Z(x_{3},t_{3})Z(x_{4},t_{4})\right] \right |^{q}\,dx_{1}dx_{2}dx_{3}dx_{4}dt_{1}\dots dt_{4}\\
& =  c\,\int_{[0,T]^{4}}\int_{(\mathbb{S}^{2})^4}\left |
\Gamma (\left\langle x_{1},x_{2}\right\rangle ,t_{2}-t_{1})\right |^{q-\alpha }
\left | \Gamma (\left\langle x_{2},x_{3}\right\rangle
,t_{3}-t_{2})\right |^{\alpha }\\
&\qquad \qquad\qquad \qquad\qquad \qquad \times \left | \Gamma (\left\langle x_{3},x_{4}\right\rangle ,t_{4}-t_{3})\right |^{q}\,dx_{1}dx_{2}dx_{3}dx_{4}dt_{1}\dots dt_{4} \\
&\leq c\,T \int_{(\mathbb{S}^{2})^4}\int_{[-T,T]}\left | \Gamma (\left\langle
x_{1},x_{2}\right\rangle ,s_{1})\right |^{q-\alpha
}ds_{1}\int_{[-T,T]}\left | \Gamma (\left\langle
x_{2},x_{3}\right\rangle ,s_{2})\right |^{\alpha
}ds_{2}\\
&\qquad \qquad\qquad \qquad\qquad \qquad\times \int_{[-T,T]}\left | \Gamma (\left\langle
x_{3},x_{4}\right\rangle ,s_{3})\right |^{q}\,ds_{3}\,dx_{1}dx_{2}dx_{3}dx_{4} \\
& \leq c\,T \int_{\lbrack -T,T]} \sum_{\ell=0}^{+\infty} \frac{(2\ell+1)}{4\pi} C_\ell(0) \left | \frac{C_\ell(s_1)}{C_\ell(0)}\right |^{q-\alpha}\, ds_{1} \int_{[-T,T]} \sum_{\ell=0}^{+\infty} \frac{(2\ell+1)}{4\pi} C_\ell(0) \left | \frac{C_\ell(s_2)}{C_\ell(0)}\right |^{\alpha}\,ds_{2}\\
&\qquad \qquad\qquad \qquad\qquad \qquad\times \int_{[-T,T]} \sum_{\ell=0}^{+\infty} \frac{(2\ell+1)}{4\pi} C_\ell(0) \left | \frac{C_\ell(s_3)}{C_\ell(0)}\right |^{q}\,ds_{3}\,, 
\end{flalign*}
where for the last inequality we used Jensen inequality, recalling \eqref{somma1}. 
For $k=1,\dots, q-1$ we have that, as $T\to +\infty$, 
$$
\int_{\lbrack -T,T]} \sum_{\ell=0}^{+\infty} \frac{(2\ell+1)}{4\pi} C_\ell(0) \left | \frac{C_\ell(\tau)}{C_\ell(0)}\right |^{k}\, d\tau = O\left ( T^{1-k\beta_{\ell^\star}}(1 +  \mathbf{1}_{k\beta_{\ell^\star}=1}\log T ) \right )
$$
whereas for $k=q$ (since $q\beta_{\ell^\star} >1$)
$$
\int_{\lbrack -T,T]} \sum_{\ell=0}^{+\infty} \frac{(2\ell+1)}{4\pi} C_\ell(0) \left | \frac{C_\ell(\tau)}{C_\ell(0)}\right |^{q}\, d\tau = O\left ( 1 \right ).
$$
Hence, as $T\to +\infty$,
\begin{eqnarray*}
&&\Cum_{4}\left ( \int_{0}^{T}\int_{\mathbb{S}^{2}}H_{q}(Z(x,t))dxdt\right )= O\left ( T^{3-q\beta_{\ell^\star}}(1 + \delta_{(q-\alpha)\beta_{\ell^\star}}^1 \log T) (1 +\delta_{\alpha\beta_{\ell^\star}}^1 \log T)       \right ) \,.
\end{eqnarray*}%
From Proposition \ref{prop-varqth}  we know that $\Var\left( \mathcal{M}_T(u)\right) \sim T\,\sum_{k=1}^{+\infty} s_k^2$ thus as $T\to +\infty$
\begin{equation*}
\Cum_{4}\left ( \frac{\mathcal{M}_T(u)[q]}{\sqrt{\Var\left( \mathcal M_{T}(u)\right) }}\right ) = O\left ( T^{1-q\beta_{\ell^\star}}(1 + \mathbf{1}_{(q-\alpha)\beta_{\ell^\star}=1} \log T) (1 +\mathbf{1}_{\alpha\beta_{\ell^\star}=1} \log T)       \right ) 
\end{equation*}
so that
$$
\lim_{T\to\infty} \Cum_{4}\left ( \frac{\mathcal{M}_T(u)[q]}{\sqrt{\Var\left( \mathcal M_{u}(T)\right) }}\right ) =0
$$
and the proof is concluded.
\end{proof}

We are now in the position to prove Theorem \ref{quattro}.

\begin{proof}[Proof of Theorem \ref{quattro}]
Here we have $\beta_0=1$. From Lemma \ref{lem1} and Condition \ref{basic2} we know that 
$$
\lim_{T\to\infty}\frac{\Var(\mathcal M_T(u)[1])}{T} = \phi(u)^2 \int_{-\infty}^{+\infty} C_0(\tau)\,d\tau >0\,.
$$
Assume first that $u\ne0$ and $2\beta _{\ell ^{\star }}>1$, then, using Propositions \ref{prop-var2nd} and \ref{prop-varqth} we have that
\begin{equation*}
\lim_{T\rightarrow \infty }\frac{\Var\left( \mathcal{M}%
_{T}(u)[2]\right) }{T\,}=\frac{u^{2}\phi (u)^{2}}{2}\sum_{\ell =0}^{\infty }(2\ell +1)\,F_{\ell},
\end{equation*}
and for $q\ge 3$, since of course $q\beta_{\ell^\star} >1$,
$$
\lim_{T\to\infty}\frac{\Var\left(\mathcal{M}_T(u)[q]\right)}{T}=s^2_q\,,
$$
where we recall that
$$
s^2_q =  \frac{4\pi\,H_{q-1}(u)^2\phi(u)^2}{q!}\sum_{\ell
_{1},\ell _{2},\dots ,\ell _{q}=0}^{\infty }\mathcal{G}_{\ell _{1}...\ell _{q}}^{0\,...0}\prod\limits_{i=1}^{q}\sqrt{\frac{2\ell _{i}+1}{4\pi }} \int_{-\infty}^{+\infty} C_{\ell _{1}}(\tau)\,C_{\ell _{2}}(\tau)\cdots C_{\ell_{q}}(\tau)\,d\tau\,.
$$
As in the proof of \eqref{eq1-proofthm1},
thanks to Dominated Convergence Theorem, we can write 
\begin{eqnarray*}
\lim_{T\to\infty}\frac{\Var\left(\mathcal{M}_T(u)\right)}{T}&=&\lim_{T\to\infty} \frac{\Var\left(\mathcal{M}_T(u)[1]\right)}{T}+\lim_{T\to\infty}\frac{\Var\left(\mathcal{M}_T(u)[2]\right)}{T}\\
&&+\sum_{q\ge3}\lim_{T\to\infty}\frac{\Var\left(\mathcal{M}_T(u)[q]\right)}{T}\\
&=&\phi(u)^2 \int_{-\infty}^{+\infty} C_0(\tau)\,d\tau+\frac{u^{2}\phi (u)^{2}}{2}\sum_{\ell =0}^{\infty }(2\ell +1)\,F_{\ell}+ \sum_{q\ge3} s^2_q\,.
\end{eqnarray*}
Now assume that $u=0$ and $3\beta _{\ell ^{\star }}>1$, then analogously
\begin{flalign*}
\lim_{T\to\infty}\frac{\Var\left(\mathcal{M}_T(u)\right)}{T}&=\lim_{T\to\infty}\frac{\Var\left(\mathcal{M}_T(u)[1]\right)}{T}+\sum_{q\ge1}\lim_{T\to\infty}\frac{\Var\left(\mathcal{M}_T(u)[2q+1]\right)}{T}\\
&=\phi(u)^2 \int_{-\infty}^{+\infty} C_0(\tau)\,d\tau+\sum_{q\ge 1} s^2_{2q+1}\,.
\end{flalign*}

In order to prove Convergence in distribution to a Gaussian random variable, we are going to check the four conditions of Lemma \ref{lemma:NP12}. Conditions \emph{(a)} and \emph{(b)} are verified thanks to 
Theorem \ref{quattro} and Proposition \ref{prop-varqth}. 
Condition \emph{(c)} of Lemma \ref{lemma:NP12} is immediately verified by Proposition \ref{prop_aus}, thanks to the Fourth Moment Theorem \cite[Theorem 1]{NP:05}.
Let us now check Condition \emph{(d)}. 
Recall that 
$$
\phi(u)^2\int_{\mathbb R}C_0(\tau) d\tau>0\,.
$$
Fix $\varepsilon>0$ such that 
$$
\phi(u)^2\int_{\mathbb R}C_0(\tau) d\tau -\varepsilon>0\,.
$$
Then we have for some $T_\varepsilon >0$ 
$$
\frac{\Var\left(\mathcal{M}_T(u)[1]\right)}{T}\ge\phi(u)^2\int_{\mathbb R}C_0(\tau) d\tau -\varepsilon,
$$
for every $T>T_\varepsilon$. 
Hence
\begin{eqnarray*}
&&\sup_{T>T_{\varepsilon}}\sum_{q=Q}^\infty\frac{\Var\left(\mathcal{M}_T(u)[q]\right)}{\Var\left(\mathcal{M}_T(u)\right)}\le
 \frac{\sup_{T>T_{\varepsilon}}\sum_{q=Q}^\infty\frac{\Var\left(\mathcal{M}_T(u)[q]\right)}{T}}{\phi(u)^2\int_{\mathbb R}C_0(\tau) d\tau -\varepsilon} \to 0\,,
\end{eqnarray*}
as $Q\to\infty$.
As a consequence, Condition \emph{(d)} of Lemma \ref{lemma:NP12} is satisfied and the proof is concluded.
\end{proof}

\subsection{Proof of Proposition \ref{mono}}

\begin{proof}
From Theorem \ref{due} we have
\begin{equation}\label{duestar}
\lim_{T\rightarrow \infty }\frac{\Var\left( \mathcal{M}_{T}(u)\right) }{%
T^{2-2\beta _{\ell ^{\star }}}}=\frac{u^2\phi (u)^{2}}{2(1-2\beta
_{\ell ^{\star }})(1-\beta _{\ell ^{\star }})}(2\ell^\star +1)C_{\ell^\star }(0)^2.
\end{equation}
Let us study the variance of $m_{T;\ell^\star}(u)$. 
\begin{eqnarray*}
\Var(m_{T;\ell^\star}(u)) &=& \frac{u^2 \phi(u/\sigma_{\ell^\star})^2}{2\sigma_{\ell^\star}^2}\int_{(\mathbb S^2)^2\times [0,T]^2} \frac{C_{\ell^\star}(t-s)^2}{C_{\ell^\star}(0)^2}  P_{\ell^\star}(\langle x, y \rangle)^2  dxdydtds\\
&=&  \frac{u^2 \phi(u/\sigma_{\ell^\star})^2}{2\sigma_{\ell^\star}^2}\int_{[0,T]^2} \frac{(4\pi)^2}{2\ell^\star+1} \frac{C_{\ell^\star}(t-s)^2}{C_{\ell^\star}(0)^2}  dtds\\
&=&  \frac{u^2 \phi(u/\sigma_{\ell^\star})^2}{2\sigma_{\ell^\star}^2} 2T\int_{[0,T]} \frac{(4\pi)^2}{2\ell^\star+1} \left ( 1 -\frac{\tau}{T}\right )\frac{C_{\ell^\star}(\tau)^2}{C_{\ell^\star}(0)^2}  d\tau\,.
\end{eqnarray*} 
From Proposition \ref{lemma-var2nd} 
\begin{equation*}
\lim_{T\rightarrow \infty }\frac{2T}{T^{2-2\beta _{\ell^\star }}}%
\int_{0}^{T}\left ( 1 -\frac{\tau}{T}\right )C_{\ell^\star}^{2}(\tau)\,d\tau=\frac{C_{\ell^\star }(0)^{2}}{(1-\beta_{\ell^\star})(1-2\beta_{\ell^\star})}
\end{equation*}
so that 
\begin{equation}\label{m2}
\lim_{T\rightarrow \infty }\frac{\Var(m_{T;\ell^\star}(u))}{T^{2-2\beta _{\ell^\star }}} = \frac{u^2 \phi(u/\sigma_{\ell^\star})^2}{2\sigma_{\ell^\star}^2} \frac{(4\pi)^2}{2\ell^\star+1} \frac{1}{(1-\beta_{\ell^\star})(1-2\beta_{\ell^\star})}.
\end{equation}
Let us now compute the covariance between $\mathcal M_T(u)$ and $m_{T;\ell^\star}(u)$: by orthogonality of Wiener chaoses
\begin{eqnarray*}
&&\Cov\left( \mathcal{M}_{T}(u),m_{T;\ell^{\star
}}(u)\right) =\Cov\left ( \mathcal{M}_{T}(u)[2], m_{T;\ell^{\star }}(u)\right ) \\
&=&\frac{J_{2}(u)\,J_{2}\left( \frac{u}{\sigma _{\ell^{\star }}%
}\right) }{4}\int_{[0,T]^{2}}\int_{\mathbb{S}^{2}\times \mathbb{S}^{2}}%
\mathbb{E}\left[ H_{2}\left( Z(x,t)\right) H_{2}\left( \frac{Z_{\ell _{\star
}}(y,s)}{\sigma _{\ell^{\star }}}\right) \right] \,dx\,dy\,dt\,ds \\
&=&\frac{J_{2}(u)\,J_{2}\left( \frac{u}{\sigma _{\ell^{\star }}%
}\right) }{2}\int_{[0,T]^{2}}\int_{\mathbb{S}^{2}\times \mathbb{S}^{2}}%
\mathbb{E}\left[ Z(x,t)\frac{Z_{\ell^{\star }}(y,s)}{\sigma _{\ell _{\star
}}}\right] ^{2}\,dx\,dy\,dt\,ds 
\end{eqnarray*}
\begin{eqnarray*}
&=&\frac{J_{2}(u)\,J_{2}\left( \frac{u}{\sigma _{\ell^{\star }}%
}\right) }{2}\int_{[0,T]^{2}}\int_{\mathbb{S}^{2}\times \mathbb{S}^{2}}%
\mathbb{E}\left[ Z_{\ell^\star}(x,t)\frac{Z_{\ell^{\star }}(y,s)}{\sigma _{\ell _{\star
}}}\right] ^{2}\,dx\,dy\,dt\,ds\\
&=&\frac{J_{2}(u)\,J_{2}\left( \frac{u}{\sigma _{\ell^{\star }}}%
\right) }{2\,\sigma _{\ell^{\star }}^{2}}\int_{[0,T]^{2}}\int_{\mathbb{S}%
^{2}\times \mathbb{S}^{2}}C_{\ell^{\star }}(t-s)^{2}\,\left( \frac{2\ell^{\star }+1}{4\pi }\right) ^{2}\,P_{\ell^{\star }}\left( \langle x,y\rangle
\right) ^{2}\,dx\,dy\,dt\,ds \\
&=&(2\ell^{\star }+1)\frac{J_{2}(u)\,J_{2}\left( \frac{u}{\sigma
_{\ell^{\star }}}\right) }{2\,\sigma _{\ell^{\star }}^{2}}%
\int_{[0,T]^{2}}\,C_{\ell^{\star }}(t-s)^{2}\,dt\,ds.
\end{eqnarray*}%
As before 
\begin{equation}\label{cov2}
\lim_{T\to\infty} \frac{\Cov\left( \mathcal{M}_{T}(u),m_{T;\ell^{\star
}}(u)\right)}{T^{2-2\beta_{\ell^\star}}} = (2\ell^{\star }+1)\frac{J_{2}(u)\,J_{2}\left( \frac{u}{\sigma
_{\ell^{\star }}}\right) }{2\,\sigma _{\ell^{\star }}^{2}}\frac{C_{\ell^\star }(0)^{2}}{(1-\beta_{\ell^\star})(1-2\beta_{\ell^\star})}\, .
\end{equation}
Plugging \eqref{duestar} and \eqref{m2} into \eqref{cov2} we get 
\begin{flalign*}
&\lim_{T\to\infty} \Corr\left( \mathcal{M}_{T}(u),m_{T;\ell^{\star
}}(u)\right) = \lim_{T\to\infty} \frac{\Cov\left( \mathcal{M}_{T}(u),m_{T;\ell^{\star
}}(u)\right)}{\sqrt{\Var(\mathcal M_T(u)) \Var(m_{T;\ell^\star}(u))}}= 1\,,
\end{flalign*} 
that concludes the proof. 
\end{proof}

\begin{appendix}

\section{Proofs of technical Lemmas}\label{app}

From Remark \ref{remComput} 
\begin{eqnarray}\label{stima1}
\int_{[0,T]^{2}}\,C_{\ell}(t-s)^{2}\,dtds &=& 2T\int_{[0,T]}\left ( 1 - \frac{\tau}{T}\right ) C_{\ell}(\tau)^{2}\,d\tau\notag \\
&\le& 2T\left ( \int_{[0,M]}C_{\ell}(\tau)^{2}\,d\tau + \int_{[M,T]}C_{\ell}(\tau)^{2}\,d\tau\right )\notag\\
&\le& 2TC_{\ell}(0)^2\left ( M + (\varepsilon +1)^2 \int_{[M,T]} g_{\beta_{	\ell}}(\tau)^2\,d\tau\right ).
\end{eqnarray}

\begin{proof}[Proof of Lemma \ref{tec2}] 
For $\beta_\ell=1$ from \eqref{stima1} we have, for $T>\max(1,M)$,
\begin{flalign*}
&\frac{1}{T}\int_{[0,T]^{2}}\,C_{\ell}(t-s)^{2}\,dtds \\
&\le 2C_\ell(0)^2\left (M + (\varepsilon+1)^2\int_{[M,T]}g_1(\tau)^{2}\,d\tau\right )\le 2C_\ell(0)^2\left (M + (\varepsilon+1)^2\int_{\mathbb R} g_1(\tau)^{2}\,d\tau\right )\\
&\le 2C_\ell(0)^2\left (M + (\varepsilon+1)^2\int_{\mathbb R}|g_1(\tau)|\,d\tau\right )= 2C_\ell(0)^2\left (M + 2\frac{(\varepsilon+1)^2}{\alpha -1}\right ),
\end{flalign*}
where we used the fact that $|g_1(\tau)|\le 1$ for every $\tau\in \mathbb R$, see \eqref{gbeta}.
\end{proof}

\begin{proof}[Proof of Lemma \ref{tec2_1}]
Let us start with \eqref{RHS}. For $\ell\in \mathcal I^\star$,  $\beta_\ell <1$ we have 
\begin{flalign*}
&\frac{1}{T^{2-2\beta _{\ell ^{\star }}}}\int_{[0,T]^{2}}\,C_{\ell}(t-s)^{2}\,dtds \le   \frac{2C_\ell(0)^2}{T^{1-2\beta _{\ell ^{\star }}}}\left ( M + (\varepsilon+1)^2\int_{[M,T]}(1 +\tau)^{-2\beta_{\ell^\star}}\,d\tau\right )\\
&=  \frac{2C_\ell(0)^2}{T^{1-2\beta _{\ell ^{\star }}}}\left ( M + \frac{(\varepsilon+1)^2}{-2\beta_{\ell^\star}+1}\left ( (1+T)^{-2\beta_{\ell^\star}+1} - (1+M)^{-2\beta_{\ell^\star}+1}  \right ) \right )\\
&\le   \frac{2C_\ell(0)^2}{T^{1-2\beta _{\ell ^{\star }}}}\left ( M + \frac{(\varepsilon+1)^2}{-2\beta_{\ell^\star}+1}(1+T)^{-2\beta_{\ell^\star}+1}   \right )\\
&\le   2C_\ell(0)^2\left ( M + \frac{(\varepsilon+1)^2}{-2\beta_{\ell^\star}+1}\left (1+\frac1M \right )^{-2\beta_{\ell^\star}+1}   \right ),
\end{flalign*} 
where for the last inequality we recall that $T>\max(1,M)$. 

Let us now prove \eqref{RHS2}. From \eqref{stima1}, for $\ell\notin \mathcal I^\star$, $\ell \ge 1$ and $\beta_\ell<1$,
\begin{eqnarray*}
\frac{1}{T^{2-2\beta _{\ell ^{\star }}}}\int_{[0,T]^{2}}\,C_{\ell}(t-s)^{2}\,dtds 
&\le &  \frac{2C_\ell(0)^2}{T^{1-2\beta _{\ell ^{\star }}}}\left ( M + (\varepsilon+1)^2\int_{[M,T]}(1 +\tau)^{-2\beta_{\ell^{\star\star}}}\,d\tau\right ).
\end{eqnarray*}
Now for $2\beta_{\ell^{\star\star}}=1$ we have 
$$
\frac{1}{T^{1-2\beta _{\ell ^{\star }}}}\int_{[M,T]}(1 +\tau)^{-2\beta_{\ell^{\star\star}}}\,d\tau \le 2 m(\beta_{\ell^\star}),
$$
while for $2\beta_{\ell^{\star\star}}<1$ we have 
$$
\frac{1}{T^{1-2\beta _{\ell ^{\star }}}}\int_{[M,T]}(1 +\tau)^{-2\beta_{\ell^{\star\star}}}\,d\tau \le \frac{1}{-2\beta _{\ell ^{\star\star }}+1}\left (1 + \frac1M      \right )^{-2\beta_{\ell^{\star\star}}+1},
$$
otherwise 
$$
\frac{1}{T^{1-2\beta _{\ell ^{\star }}}}\int_{[M,T]}(1 +\tau)^{-2\beta_{\ell^{\star\star}}}\,d\tau \le \frac{1}{2\beta _{\ell ^{\star\star }}-1}\left (\frac{1}{1+M}    \right )^{2\beta_{\ell^{\star\star}}-1}\,,
$$
which concludes the proof.
\end{proof}

Let us write 
\begin{align}\label{qq}
&k_{\ell_1\dots \ell_q}(T)= 2T\int_{[0,T]} \left (1 -\frac{\tau}{T} \right ) C_{\ell_1}(\tau)\cdots C_{\ell_q}(\tau)\,d\tau\notag\\
&= 2T\left ( \int_{[0,M]} \left (1 -\frac{\tau}{T} \right ) C_{\ell_1}(\tau)\cdots C_{\ell_q}(\tau)\,d\tau + \int_{[M,T]} \left (1 -\frac{\tau}{T} \right ) C_{\ell_1}(\tau)\cdots C_{\ell_q}(\tau)\,d\tau\right ).
\end{align}

\begin{proof} [Proof of Lemma \ref{lemmaqq}] The proof is similar to the proof of Lemma \ref{tec2}. 
Assume that there is at least one index $j\in \lbrace 1,\dots,q\rbrace$ such that 
$\beta_{\ell_j}=1$. Let $U:=U(\ell_1,\dots,\ell_q) =\lbrace j\in \lbrace 1,\dots,q \rbrace : \beta_{\ell_j}=1\rbrace$.  We have, from \eqref{qq}, for $T>\max(1,M)$,
\begin{eqnarray*}
\frac{k_{\ell_1\dots \ell_q}(T)}{T} &\le&2C_{\ell_1}(0)\dots C_{\ell_q}(0)\Bigg ( M + (\varepsilon+1)^q \int_{[M,T]}  (1+\tau)^{-\#U\alpha -(\beta_{\ell_{1}}+\dots \beta_{\ell_{q}}-\#U)}\,d\tau\Bigg )\\
&\le&2C_{\ell_1}(0)\dots C_{\ell_q}(0)\left ( M + (\varepsilon+1)^q \int_{[M,T]}  (1+\tau)^{-\#U\alpha -(q-\#U)\min(\beta_0,\beta_{\ell^\star})}\,d\tau\right )\\
&\le& 2C_{\ell_1}(0)\dots C_{\ell_q}(0)\Bigg ( M + \frac{(\varepsilon+1)^q}{(q-\#U)\min(\beta_0,\beta_{\ell^\star}) + \#U \alpha -1}\cr
&&\times  \left (  \frac{1}{1+M}   \right)^{(q-\#U)\min(\beta_0,\beta_{\ell^\star})+ \#U \alpha -1}\Bigg )\notag \\
&\le& 2C_{\ell_1}(0)\dots C_{\ell_q}(0)\Bigg ( M + \frac{1}{q\min(\beta_0,\beta_{\ell^\star}) + \#U (\alpha-\min(\beta_0,\beta_{\ell^\star})) -1} \cr
&&\times \left (  \frac{\varepsilon+1}{(1+M)^{\min(\beta_0,\beta_{\ell^\star})}}  \right)^q \left (  \frac{1}{1+M}   \right)^{\#U (\alpha- \min(\beta_0,\beta_{\ell^\star}))-1}\Bigg )\notag \\
&\le& 2C_{\ell_1}(0)\dots C_{\ell_q}(0)\left ( M + \frac{1}{q\min(\beta_0,\beta_{\ell^\star})} \left (  \frac{\varepsilon+1}{(1+M)^{\min(\beta_0,\beta_{\ell^\star})}}  \right)^q\right)\,,
\end{eqnarray*}
which concludes the proof.
\end{proof}

\begin{proof}[Proof of Lemma \ref{tecq_1}] The proof is similar to the proof of Lemma \ref{tec2_1}. 
For $\ell_1,\dots,\ell_q\in \mathcal I^\star$,  $\beta_{\ell_j}<1$ for every $j$ and $T>\max(1,M)$, from \eqref{qq}, we have
\begin{eqnarray*}
\frac{k_{\ell_1\dots \ell_q}(T)}{T^{2-q\beta_{\ell^\star}}} &\le& \frac{2C_{\ell_1}(0)\dots C_{\ell_q}(0)}{T^{1-q\beta_{\ell^\star}}}\Big (M + (\varepsilon +1)^q\int_{[M,T]} (1+\tau)^{-q\beta_{\ell^\star}}\,d\tau\Big )\\
&\le& 2C_{\ell_1}(0)\dots C_{\ell_q}(0)\Big (M + \frac{(\varepsilon +1)^q}{1-q\beta_{\ell^\star}}\left ( 1 + \frac1M    \right )^{1-q\beta_{\ell^\star}}\Big ).
\end{eqnarray*} 
Finally for $(\ell_1,\dots,\ell_q)\notin \mathcal I^\star$, $\ell_j\ge 1$, $\beta_{\ell_j}<1$ for every $j$ and $T>\max(1,M)$, from \eqref{qq} (note that
$\beta_{\ell_1}+\dots \beta_{\ell_q} \ge \beta_{\ell^{\star\star}} + (q-1) \beta_{\ell^\star} > q\beta_{\ell^\star})$, we have
\begin{eqnarray*}
\frac{k_{\ell_1\dots \ell_q}(T)}{T^{2-q\beta_{\ell^\star}}} &\le& \frac{2C_{\ell_1}(0)\dots C_{\ell_q}(0)}{T^{1-q\beta_{\ell^\star}}}\Big (M + (\varepsilon +1)^q\int_{[M,T]} (1+\tau)^{- (\beta_{\ell^{\star\star}} + (q-1) \beta_{\ell^\star})}\,d\tau\Big )\\
&\le& 2C_{\ell_1}(0)\dots C_{\ell_q}(0)\Big (M + (\varepsilon +1)^q\frac{1}{T^{1-q\beta_{\ell^\star}}}\int_{[M,T]} (1+\tau)^{- (\beta_{\ell^{\star\star}} + (q-1) \beta_{\ell^\star})}\,d\tau\Big ).
\end{eqnarray*} 
Now if $\beta_{\ell^{\star\star}} + (q-1) \beta_{\ell^\star}<1$ then 
$$
\frac{1}{T^{1-q\beta_{\ell^\star}}}\int_{[M,T]} (1+\tau)^{- (\beta_{\ell^{\star\star}} + (q-1) \beta_{\ell^\star})}\,d\tau\le \frac{\left (1 + \frac1M      \right )^{-(\beta_{\ell^{\star\star}} + (q-1) \beta_{\ell^\star})+1}}{-(\beta_{\ell^{\star\star}} + (q-1) \beta_{\ell^\star})+1}\,;
$$
if $\beta_{\ell^{\star\star}} + (q-1) \beta_{\ell^\star}=1$, then
$$
\frac{1}{T^{1-q\beta_{\ell^\star}}}\int_{[M,T]} (1+\tau)^{- 1}\,d\tau\le \frac{1}{T^{1-2\beta_{\ell^\star}}}\int_{[M,T]} (1+\tau)^{- 1}\,d\tau \le 2 m(\beta_{\ell^\star}),
$$
where $m(\beta_{\ell^\star})$ is a constant defined in Lemma \ref{tec2}. Finally if $\beta_{\ell^{\star\star}} + (q-1) \beta_{\ell^\star}>1$ then
$$
\frac{1}{T^{1-q\beta_{\ell^\star}}}\int_{[M,T]} (1+\tau)^{- (\beta_{\ell^{\star\star}} + (q-1) \beta_{\ell^\star})}\,d\tau\le \frac{(M+1)^{-(\beta_{\ell^{\star\star}} + (q-1) \beta_{\ell^\star}-1)}}{\beta_{\ell^{\star\star}} + (q-1) \beta_{\ell^\star}-1} 
$$
and the proof is concluded.
\end{proof}

\end{appendix}

\noindent {\sc Dipartimento di Matematica, Universit\`a degli Studi di Roma ``Tor Vergata"}

\noindent \emph{E-mail address:} \texttt{marinucc@mat.uniroma2.it} 

\vspace{.2cm}

\noindent {\sc Dipartimento di Matematica e Applicazioni, Universit\`a degli Studi di Milano-Bicocca}

\noindent \emph{E-mail address:} \texttt{maurizia.rossi@unimib.it} 

\vspace{.2cm}

\noindent {\sc Dipartimento di Economia, Universit\`a degli Studi ``G. D'Annunzio" Chieti-Pescara}

\noindent \emph{E-mail address:} \texttt{anna.vidotto@unich.it} (corresponding author)


\begin{thebibliography}{99}
\bibitem[1]{adlertaylor} R.~J. Adler, J.~E.Taylor (2007) \textit{Random Fields
and Geometry}. Springer Monographs in Mathematics, Springer.
\vspace{-.25cm}
\bibitem[2]{AW:09} J.-~M. Az\"ais, Wschebor, M. (2009) \textit{Level Sets and Extrema of Random Processes and Fields}. Wiley.
\vspace{-.25cm}
\bibitem[3]{BP17} C. Berg, E. Porcu (2017) From Schoenberg coefficients to Schoenberg functions. \textit{Constructive Approximation}, 45, 2, 217--241.
\vspace{-.25cm}
\bibitem[4]{Ber02} M.~V. Berry (2002) Statistics of nodal lines and points in chaotic quantum billiards: perimeter corrections, fluctuations, curvature. \textit{Journal of Physics A} 35(13), 3025--3038.
\vspace{-.25cm}
\bibitem[5]{Bin72} N.~H. Bingham (1972) A Tauberian Theorem for Integral Transforms of Hankel Type, \textit{Journal of the London Mathematical Society}, 2-5, 3, 493--503.
\vspace{-.25cm}
\bibitem[6]{BS19} N.~H. Bingham, T.~L. Symons (2019) Gaussian random fields on the sphere and sphere cross line. \textit{Stochastic Processes and their Applications}, in press.
\vspace{-.25cm}
\bibitem[7]{Cam19} V. Cammarota (2019) Nodal area distribution for arithmetic random waves. \textit{Transactions of the American Mathematical Society} 372,  5, 3539--3564.
\vspace{-.25cm}
\bibitem[8]{CM19} V. Cammarota, D. Marinucci (2019) On the Correlation of Critical Points and Angular Trispectrum for Random Spherical Harmonics. \textit{Preprint arXiv:1907.05810}.
\vspace{-.25cm}
\bibitem[9]{CM2018} V. Cammarota, D. Marinucci (2018) A Quantitative Central
Limit Theorem for the Euler-Poincar\'{e} Characteristic of Random Spherical
Eigenfunctions, \textit{Annals of Probability}, 46, 6, 3188--3288.
\vspace{-.25cm}
\bibitem[10]{CD:12} G. Chiles, D\'elfiner (2012) \textit{Geostatistics}. Wiley Series in Probability and Statistics.
\vspace{-.25cm}
\bibitem[11]{Ch:05} G. Christakos (2005) \textit{Random Field Models in Earth Sciences}. Academic Press.
\vspace{-.25cm}
\bibitem[12]{DM79} R.~L. Dobrushin, P. Major  (1979) Non-Central Limit Theorems for Non-Linear Functionals of Gaussian Fields. \textit{Z. Wahrscheinlichkeitstheorie verw. Gebiete}  50, 27--52.
\vspace{-.25cm}
\bibitem[13]{DehTaq} H. Dehling and M. Taqqu (1989) The empirical process of some
long-range dependent sequences with an application to U-statistics, \emph{%
Annals of Statistics} 17, 4,1767--1783.
\vspace{-.25cm}
\bibitem[14]{Ess42} C.~G. Esseen (1942) On the Liapunoff limit of error in the theory of probability. \textit{Arkiv f\"or Matematik, Astronomi och Fysik} A28, 1--19.
\vspace{-.25cm}
\bibitem[15]{KKW13} M. Krishnapur, P. Kurlberg, I. Wigman (2013) Nodal length fluctuations for arithmetic random waves. \textit{Annals of Mathematics} 177, 2, 699--737.
\vspace{-.25cm}
\bibitem[16]{IPV95} P. Imkeller, V. P\'erez-Abreu, J. Vives (1995) Chaos expansions of double intersection local time of Brownian motion in $\mathbb R^q$ and renormalization. 
\textit{Stochastic Processes and their Applications}, 56(1):1--34.
\vspace{-.25cm}
\bibitem[17]{LO13} N. Leonenko, A. Olenko (2013) Tauberian and Abelian Theorems for Long-range Dependent Random Fields. \textit{Methodology and Computing in Applied Probability}, 15, 715--742.
\vspace{-.25cm}
\bibitem[18]{LR-M} N. N. Leonenko, M. D. Ruiz-Medina (2018) Increasing domain asymptotics for the first Minkowski functional of spherical random fields. \textit{Theory of Probability and Mathematical Statistics} 97, 127--149.
\vspace{-.25cm}
\bibitem[19]{MM18} C. Ma, A. Malyarenko (2020) Time-Varying Isotropic Vector Random Fields on Compact Two-Point Homogeneous Spaces \textit{Journal of Theoretical Probability}  33, 319--339.
\vspace{-.25cm}
\bibitem[20]{MaPeCUP} D. Marinucci, G. Peccati (2011) \textit{Random Fields on the
Sphere: Representations, Limit Theorems and Cosmological Applications},
Cambridge University Press.
\vspace{-.25cm}
\bibitem[21]{MPRW16}  D. Marinucci, G. Peccati, M. Rossi, I. Wigman (2016) Non-universality of nodal length distribution for arithmetic random waves. \textit{Geometric and Functional Analysis} 26, 3, 926--960.
\vspace{-.7cm}
\bibitem[22]{MR2015} D. Marinucci, M. Rossi (2015) Stein-Malliavin approximations
for nonlinear functionals of random eigenfunctions on $\mathbb S^{d}$, \textit{Journal of Functional Analysis}, 268, 8, 2379--2420.
\vspace{-.25cm}
\bibitem[23]{MRW} D. Marinucci, M. Rossi, I. Wigman (2020) The asymptotic
equivalence of the sample trispectrum and the nodal length for random
spherical harmonics. \textit{Annales de l'Institut Henri Poincar\'e, Probabilit\'es et Statistiques}, 56, 1, 374--390.
\vspace{-.25cm}
\bibitem[24]{MV16} D. Marinucci, S. Vadlamani (2016) High-frequency asymptotics of Lipschitz-Killing curvatures of excursion sets on the sphere. \textit{The Annals of Applied Probability} 26, 1, 462--506.
 \vspace{-.7cm}
\bibitem[25]{MW2014} D. Marinucci, I. Wigman (2014) On nonlinear functionals of
random spherical eigenfunctions, \textit{Communications in Mathematical
Physics}, 327, 3, 849--872.
\vspace{-.25cm}
\bibitem[26]{KN:17} G.~R. North and K.~-Y. Kim. (2017) \textit{Energy Balance Climate Models}. Wiley Series in Atmospheric Physics and Remote Sensing.
\vspace{-.25cm}
\bibitem[27]{noupebook} I. Nourdin, G. Peccati (2012) \textit{Normal
Approximations Using Malliavin Calculus: from Stein's Method to Universality}, Cambridge University Press.
\vspace{-.25cm}
\bibitem[28]{NPR19} I. Nourdin, G. Peccati, M. Rossi (2019) Nodal statistics of planar random waves. \textit{Communications in Mathematical Physics} 369, 1, 99--151.
\vspace{-.25cm}
\bibitem[29]{NP:05} D. Nualart and G. Peccati (2005) Central limit theorems for sequences of multiple stochastic integrals. \textit{Annals of Probability} 33, 1, 177-193.
\vspace{-.25cm}
\bibitem[30]{PV20} G. Peccati, A. Vidotto (2020) Gaussian Random Measures Generated by Berry's Nodal Sets. \textit{Journal of Statistical Physics} 178, 4, 996--1027.
 \vspace{-.25cm}
\bibitem[31]{Rossi2018} M. Rossi (2019) Random nodal lengths and Wiener chaos,
\textit{Probabilistic Methods in Geometry, Topology and Spectral Theory}, Contemporary Mathematics Series, 739, 155--169.
\vspace{-.25cm}
\bibitem[32]{Sze75} G. Szeg\"o (1975) \emph{Orthogonal polynomials}, vol. XXIII, 4th edn. American Mathematical Society, Providence, RI.
\vspace{-.25cm}
\bibitem[33]{Ta:79} M. Taqqu (1979) Convergence of Integrated Processes of Arbitrary Hermite Rank, \textit{ Z. Wahrscheinlichkeitstheorie verw. Gebiete}, 50, 53--83.
\vspace{-.25cm}
\bibitem[34]{Todino1} A.~P. Todino (2019) A Quantitative Central Limit Theorem for
the Excursion Area of Random Spherical Harmonics over Subdomains of $\mathbb S^2$, \textit{Journal of Mathematical Physics} 60, 023505.
 \vspace{-.25cm}
\bibitem[35]{VT13} M.~K. Veillette and M.~S. Taqqu (2013) Properties and numerical evaluation of the Rosenblatt distribution. \textit{Bernoulli} 19, 3,  982--1005.
\vspace{-.25cm}
\bibitem[36]{Wig} I. Wigman (2010) Fluctuations of the nodal length of random
spherical harmonics, \textit{Communications in Mathematical Physics}, 298,
3, 787--831.
\end{thebibliography}
\end{document}